\documentclass[a4paper, 12pt]{article}
\usepackage{graphics}
\usepackage{ifthen}
\usepackage{amsmath,amsfonts,amstext,amscd,bezier,amssymb,a4,enumerate,epsfig,psfrag,subfigure}
\usepackage[latin1]{inputenc}
\usepackage{graphicx}
\usepackage{amsthm,amsopn}
\usepackage{bm}
\usepackage{amssymb}
\usepackage[T1]{fontenc}
\usepackage{mathrsfs}

\usepackage{color}

\everymath={\displaystyle}

\newcommand{\Fb}{{\bf F}}
\newcommand{\Lb}{{\bf L}}
\newcommand{\Pb}{{\bf P}}
\newcommand{\fb}{{\bf f}}
\newcommand{\Gb}{{\bf G}}
\newcommand{\ab}{{\bf a}}
\newcommand{\bb}{{\bf b}}

\newcommand{\NN}{{\mathbb{N}}}

\newcommand{\RR}{{\mathbb{R}}}
\newcommand{\CC}{{\mathbb{C}}}

\newcommand{\BY}{{{\cal B}(Y)}}
\newcommand{\BH}{{{\cal B}(H)}}

\newcommand{\Hab}{{\mathscr P}^{b}}
\newcommand{\HbC}{{\mathscr P}^{b, p}}
\newcommand{\Haz}{{\mathscr P}^{0}}
\newcommand{\Haone}{{\mathscr P}^{b}}
\newcommand{\Habk}{{\mathscr P}^{b_k}}
\newcommand{\Habb}{{\mathscr P}^{{\bf b}}}

\newcommand{\Qab}{{\mathscr Q}^{ab}}

\newcommand{\Cr}{{\mathcal C}}

\newcommand{\Ftl}{{F^t}}

\newcommand{\Ftkr}{{F^{t_k}}}

\newcommand{\sconv}{{ \stackrel{{\scriptstyle s}}{\longrightarrow}\ }}
\newcommand{\wconv}{{ \stackrel{{\scriptstyle w}}{\longrightarrow}\ }}

\newcommand\parity{\operatorname{par}}
\newcommand\ind{\operatorname{ind}}

\newcommand\sgn{\operatorname{sgn}}

\newtheorem{lemma} {Lemma}
\newtheorem{prop} {Proposition}
\newtheorem{theo} {Theorem}

\newtheorem{cor} {Corollary}

\newtheorem*{theo*} {Theorem}

\renewcommand{\qed}{\hfill \mbox{\raggedright \rule{.1in}{.1in}}}

\newcommand{\Ran}{\operatorname{Ran}}
\newcommand{\Dom}{\operatorname{Dom}}
\newcommand{\Ker}{\operatorname{Ker}}

\newcommand{\trace}{\operatorname{tr}}

\title{Positive eigenvectors and simple nonlinear maps}

\author{Marta Calanchi and Carlos Tomei}

\date{}

\begin{document}
\maketitle

\begin{abstract}
For linear operators $L, T$ and  nonlinear  maps $P$, we describe classes of simple maps $F = I - P T$, $F = L - P$ between Banach and Hilbert spaces, for which no point has more than two preimages. The classes encompass
known examples (homeomorphisms, global folds) and the weaker, geometric, hypotheses suggest new ones. The operator
$L$ may be the Laplacian with various  boundary conditions, as in the original Ambrosetti-Prodi theorem, or the operators associated with the quantum harmonic oscillator, the hydrogen atom, a spectral fractional Laplacian, elliptic operators in non-divergent form. The maps $P$ include the Nemitskii map  $P(u) = f(u)$ but may be non-local, even non-variational. For self-adjoint operators $L$, we employ familiar results on the nondegeneracy of the ground state. On Banach spaces, we use a variation of the Krein-Rutman theorem.
\end{abstract}

\medbreak


\smallbreak

{\noindent\bf MSC-class:} 35J65, 47H11, 47H30, 58K05.

\bigskip

\section{Introduction and main results}\label{sec:intro}

The celebrated Ambrosetti-Prodi theorem provides a geometric description of a class of differential maps. We state it in the context of Sobolev spaces (\cite{A}, \cite{AP}, \cite{BP}). For a smooth, bounded domain $\Omega \subset \RR^n$, denote by $ \lambda_m <\ \mu_m $ the two smallest eigenvalues of the (Dirichlet) Laplacian
$ - \Delta_D: D = H^2(\Omega)\cap H^1_0(\Omega) \to L^2(\Omega) $ and by $\phi_m >0$ an eigenvector associated with $\lambda_m$.

\begin{theo} [Ambrosetti-Prodi] \label{AP} Let $\Fb: D \to L^2(\Omega), \Fb(u) = - \Delta_D u - \fb(u)$, where
$\fb: \RR \to \RR$ is a $C^2$ strictly convex  function  for which
\[ \ab = \inf \fb ' (x) <  \lambda_m < \lambda_m + \bb = \sup \fb '(x) < \mu_m  \ . \]
Then $\Fb$ folds downwards with respect to $\phi_m$.
\end{theo}

We define folds. Given real Banach spaces $X$ and $Y$ and $\phi \in X \cap Y, \phi \ne 0$, let $V = \langle \phi \rangle$ be the subspace generated by $\phi$. Denote  the {\it vertical} line through $y \in Y$ by $L_y =\{ y + h \phi, h \in \RR\} \subset Y$. A continuous function $F:X \to Y$ is {\it simple} if there is a subspace decomposition $X = W_X \oplus V$ with the following property: for each $y \in Y$, the inverse $F^{-1}(L_y)$ is a curve ( a fiber)
\[ \Lambda_y = \{u(y,t) = w(y,t) + t \phi, w(y,t) \in W_X, t \in \RR \} \]
such that, for $y + h(y,t)\phi = F(u(y,t))$, the height $h(y,.): \RR \to \RR$ is either strictly monotonical or strictly unimodal for each $y \in Y$. Thus, no point has three preimages under $F$. The definition clearly does not depend on the choice of $W_X$.

The function $F$ {\it folds downwards} with respect to $\phi$ if additionally, for all $y \in Y$, $h(y,t) \to -\infty$ as $t \to \pm\infty$. Then, there is a height $h(y)$ such that $F(u) = y + h \phi$ has zero, one or two preimages, depending if $h > h(y), h = h(y)$ or $ h < h(y)$. Similarly, $F$ {\it folds upwards} if the inequalities are reversed and  $F$ folds {\it folds vertically} if either case happens.

\medskip
Manes and Micheletti  stated the theorem with the hypotheses  above in the context of Hölder spaces (\cite{MM}). With different techniques, Berger and Podolak  provided a different proof (\cite{BP}, also \cite{BC}) for Hilbert spaces. The theorem was a starting point of an active line of research, leading to characterizations of folds and further examples.  The interested reader may find a rich  collection of examples and techniques in the review papers by Church and Timourian (\cite{ChT}, \cite{CT1}),
and  Ruf (\cite{R2}),  which cover extremely well the material up to the mid nineties. Their approach is strongly influenced by the original Ambrosetti-Prodi view of the problem:  a combination of local  theory (the fold being the first, simplest singularity) with global information related to the determination of the number of solutions in terms  of some parameter or forcing terms.
Such results are frequently restricted by technical requirements, such as smoothness: in this paper, we present operational alternatives and provide generalizations to a wider class of problems.

\medskip
Recently, Sirakov et alii. (\cite{STZ}) considered folds given by nonlinear perturbations of second order operators $\Lb$ in non-divergence form. For a bounded $C^{1,1}$ domain $\Omega \subset \RR^n$, $n \ge 2$,  $X = W^{2,n}(\Omega) \cap W^{1,n}_0(\Omega) $, $Y=W^{0,n}(\Omega) = L^n(\Omega)$, set
\[ \Lb u = - A_{ij} \ \partial_i \partial_j u - bB_i \ \partial_i u - q \ u = - \trace (A D^2 \ u)- B \  \nabla u - q\  u, \]
where, for appropriate constants $\Lambda \ge \lambda >0$,
the matrix $A = (A_{ij})$ satisfies $A(x) \in C(\overline{\Omega})$, $\sigma(A(x)) \subset [\lambda, \Lambda]$ and $|B|_\infty, |q|_\infty \le C$. Then $\Lb \in {\cal B}(X,Y)$, the space of bounded linear operators from $X$ to $Y$. Let $\lambda_m = \lambda_m(\Lb)$ be the eigenvalue of $\Lb$ of smallest real part and $\phi_m >0$ an associated eigenvector. 

\begin{theo} [Sirakov-Tomei-Zaccur] \label{theo:BNV}
Let $\Lb$ as above and $\Pb(u)= \fb(u)$ for a strictly convex function $\fb: \RR \to \RR$ such that
\[ \ab \  = \ \inf_{r \ne s} \frac{f(r) - f(s)}{r-s} \ < \  \lambda_m \ < \ \lambda_m + \bb \ = \ \sup_{r \ne s} \frac{f(r) - f(s)}{r-s} \ . \]
Then, for some $ \bb > 0$, $\Fb = \Lb - \Pb:X\to Y$ folds downwards with respect to $\phi_m$.
\end{theo}

The proof in $\cite{STZ}$ combines elliptic estimates and geometric arguments. We phrase an extension, Theorem \ref{theo:BNV2}, as a geometric statement with a geometric proof.

\medskip

Both results are inspirational, in the specification of hypotheses and in the techniques of proof. Our results are more geometric and more general. We take a different point of view: simple maps are obtained from nonlinear perturbationsof special linear operators. As for the theorems above, the positivity of an eigenvector of the Jacobian, or some appropriate linearization, is an essential ingredient. We use the preservation of positive eigenfunctions associated with the spectral radius (Krein-Rutman type arguments, as in Section 2), or with the smallest eigenvalue of a self-adjoint map (the perturbation theory of nondegenerate groundstates, Section  3). On a global scale, the asymptotic linearity of the nonlinear terms yields a certain uniformity of the linearizations, which belong to a set with privileged compactness properties.

\medskip
Let $Y$ be a real, separable, reflexive Banach space endowed with a normal, generating cone $K \subset Y$ (Section \ref{Cones} presents the terminology associated with cones). Denote  by $k \ge 0$ an element $k \in K$ and by $k > 0$ a quasi-interior point $k \in K$. For a bounded operator $S: Y \to Y$, (i.e. $S \in \BY$), denote its spectrum and spectral radius by $\sigma(S)$ and $r(S)$. An  operator $S \in \BY$ is {\it r-special} if it compact and ergodic with respect to $K$ and $r(S)>0$.

\medskip
Consider a r-special operator $T \in \BY$ with $r(T)=1$. For $b, R >0 $ and an r-special operator $ S\in \BY$, define the set of perturbations
\[  \Hab \ = \  {\mathscr P}^{R,b,S}_T \ = \{   H =  A +  B , \ \hbox{for} \ \ A, B \in \BY, \  \| A\|\le R,  \ \|  B \| \le b , \]\[ \  S \le  A  T \le  T, \   B \ge 0 \ \} \ . \]

\medskip
The map $P: Y \to Y$ admits {\it linearizations} if for $y,z \in Y$, there exists $G(y,z) = G(z,y) \in \BY$ for which $P(y) - P(z) =  G(y,z)(y-z)$.

The proof of Theorem \ref{AP} in \cite{AP} uses linearizations. Linearizations $G(y,z) \in \BY$ are not necessarily continuous in $y,z \in Y$.

\medskip
We enumerate possible properties of $P: Y \to Y$. Let $F: Y \to Y, F = I - PT$.
\begin{enumerate}
\item[(r-H)] $P: Y \to Y$ is a Lipschitz map admitting linearizations $G(y,z) \in \Hab$.
\item [(r-Conv)]
     Let $y_1, y_2, z_1, z_2 \in Y$ such that $y_1 > z_1$, $y_2 > z_2 $, $y_1 \ge y_2 $, $z_1 \ge z_2$ and either $y_1 > y_2 $ or $z_1 > z_2 $. Then $ G(y_1,z_1)  -  G(y_2,z_2) \in \BY$ is  strictly positive.
     \item[(r-Hs)] $P: Y \to Y$ is a $C^1$ map whose Jacobians $J(y)=DP(y)$ belong to $\Hab$.
\item [(r-Convs)]
     Let $y, z \in Y$, $y > z$. Then $ DP(y) - DP(z) \in \BY$ is  strictly positive.
\item[(r-Crit)] $P: Y \to Y$ is $C^1$ and some $y \in Y$ is {\it critical}: $0 \in \sigma(DF(y))$.
\end{enumerate}

\begin{theo}\label{theo:F} Suppose  $Y$ and $K$ as above, $T  \in \BY$  an r-special operator for which $r(T)=1, T \phi = \phi > 0$ and $P:Y\to Y$ a Lipschitz map.  Define the function $F:Y\to Y,  y \mapsto y - P(Ty)$. For small $b>0$,

\smallskip
\noindent (1)  suppose  that (r-H) holds and, for $y,z \in Y$, $r(G(Ty,Tz)T) \ne 1$ .   Then $F$ is a homeomorphism onto its image.

\smallskip
\noindent (2) suppose  that (r-H), (r-Conv) hold. Then $F$ is a simple map.

\smallskip
\noindent (3) if (r-H), (r-Conv) hold and $F: Y \to Y$ is proper, then it is a homeomorphism or folds vertically.

\smallskip
\noindent (4) suppose that (r-Hs), (r-Convs) and (r-Crit) hold. Then $F$ folds downwards.
\end{theo}

Item (1) is a non-symmetric version of the Dolph-Hammerstein theorem (\cite{D}, \cite{H}). In items (2) and (3), we circumvent the arguments in \cite{AP} and \cite{BP} which, under smoothness hypotheses, identify the critical points of $F$ as being folds, a fact which follows from the theorem. For Nemitskii maps $P(u) = f(u)$, we obtain finer results (Theorem \ref{theo:BNV2}), from which Theorem \ref{BNV} follows.
Theorem \ref{theo:F} admits variations, but we are interested in displaying a set of techniques, more than specific results. Theorem \ref{theo:F} (4), for example, is extended in Section \ref{newfolds}.

\smallskip

\smallskip
Here is a sketch of proof. From (r-H) or (r-Hs), linearizations (or Jacobians) belong to $\Hab$. We begin with a perturbation result: for $H \in \Hab$ and a special operator $T \in \BY$, $HT$ is also r-special. We use an extension of the Krein-Rutman theorem for cones with possibly empty interior (\cite{De}; for an alternative approach, \cite{Bi}). It is a diluted version of Theorem \ref{MSS} in the Appendix, a condensation of results by Marek, Bonsall, Schaefer and Sawashima (\cite{Mk}).

The convex set $\Hab$ has good compactness properties, implying a uniform coercive bound for operators $I - HT$. An argument with covering spaces, the Banach-Mazur theorem, then converts $F$ into {\it adapted coordinates} (Proposition \ref{adaptedcoordinatesFr}) $F^a : Y \to Y$ for $F$, a nonlinear rank one perturbation of the identity, already present in the original proof of Theorem \ref{AP} by Berger and Podolak (\cite{BP}) and in the proof of Theorem \ref{BNV} in \cite{STZ}. At this point, the proof of Theorem \ref{theo:F}(2) is at hand.

In combination with the convexity-like hypothesis (r-Conv), no point in the counterdomain has more than two preimages under $F$ (Proposition \ref{NT}), so that $F^a$ and $F:Y \to Y$  are simple maps (Theorem \ref{theo:F}(2)). In Theorem \ref{AP}, the convexity of $f$ is essentially {\it necessary}  for $F(u)= - \Delta u - f(u)$ to be a fold (\cite{CTZ2}): the search for further examples led us to  the  material in this text.

From the adapted coordinates, the inverse under $F$ of a vertical line $L$ is a {\it fiber} $\Lambda_y$: our geometric description of $F$ is complete once we compute the asymptotic behavior of $F$ at endpoints of the fiber, informally, how $F(\Lambda_y)$ goes up and down along $L$. Here the arguments bifurcate. For Nemitskii maps, a standard argument (Theorem \ref{theo:BNV2}) provides our proof to Theorem \ref{theo:BNV}.
For more general $C^1$, the existence of one critical point of $F$ (hypothesis (m-Crit)) yields a full fledged critical set (Corollary \ref{cor:abundantcritical}), from which the properness of $F$ follows (Proposition \ref{properl}), and the asymptotic behavior is restricted. Finally, the fact that the images of fibers fold downwards is a consequence of degree theory (Proposition \ref{index}).

Inspired by the fact that Theorem \ref{AP} holds for Sobolev and Hölder spaces, we present a set up in Section \ref{newfolds} where folds transfer between scales of spaces.

\medskip
In Section \ref{section:self-adjoint}, we consider additive, nonlinear perturbations $F(u) = L u - P(u)$ of a self-adjoint operator $L: D \subset H \to H$ on Hilbert spaces $H$ of functions, possibly defined on unbounded domains. We concentrate on the smallest eigenvalue $\lambda_m$ of $L$, which we require to be {\it basic}, an isolated point in $\sigma(L)$, whose invariant subspace is generated by a positive eigenfunction $\phi > 0$. As in Theorem \ref{AP}, the corresponding Theorem \ref{theo:F2} does not restrict the analogous parameter $b$.

Again, the appropriate perturbation theory, from  Faris and Simon (\cite{F}, \cite{RS}), delimits a set containing linearizations and Jacobians. The estimates are tighter, being of a spectral nature, and allow for a different technique to derive properness of $F$, which we apply to spaces of functions on unbounded sets. The compactness of resolvents $(L - \gamma I)^{-1}$ is not needed.

We complete the text with some examples in Section  \ref{sec:examples}.

\medskip
{\bf Acknowledgements:}
Calanchi is partially supported by INdAM-GNAMPA Project 2018.
Tomei is supported by CAPES, CNPq and FAPERJ. He thanks Boyan Sirakov for a number of illuminating conversations.

\section{Proof of Theorem \ref{theo:F}} \label{BNV}

\subsection{Cones and the spectral radius of r-special operators} \label{Cones}

We recall some basic facts about  cones (\cite{De}, \cite{Mk}, \cite{GCZ}).
Let $Y$ be a real Banach space. We denote its norm by $|\cdot|$.
A {\it cone} $K \subset Y$ is a  {\it closed} set for which
\[ 0 \in K \ , \quad K + K \subset K \ , \quad t \ge 0 \ \Rightarrow \ tK \subset K , \quad K \cap (-K) = \{0\} \ . \]
\noindent It induces a partial order in $Y$: $y \le z \Leftrightarrow z - y \in K$. A cone $K$ is {\it reproducing} if $K - K = Y$.
A cone is {\it normal} if, for some $\alpha \in \RR$,
$0\le y\le z \Longrightarrow |y|\le \alpha| z|$.
For a reproducing cone $K$, the {\it dual cone} is
 $K^\ast = \{ k^\ast \in Y^\ast \ | \ \langle k^\ast , k \rangle \ge 0 ,  \forall k \in K \} $
 (brackets denote  the coupling between $Y$ and $Y^\ast$).
According to a result of Krein  (\cite{GCZ}, \cite{Mk}), $K$ is  a normal,  reproducing cone if and only if $K^*$ also is.

\smallskip
An element $k \in K$ is  {\it quasi-interior} if $\langle k^\ast, k\rangle\neq 0$,
for any $k^\ast \in K^\ast \setminus \{0\}$. Similarly, a linear functional $k^\ast \in K^\ast$ is {\it strictly positive} if $\langle k^\ast, k \rangle > 0$  for all $k \in K \setminus \{0\}$. To simplify notation, we denote a quasi-interior point $k$ by $k>0$ and a strictly positive functional $k^\ast$ by $k^\ast >0$.

\smallskip
A cone $K$ is {\it unflattened} if, for some  $M>0$ and any $y \in Y$, there are $k_1, k_2 \in K$ such that $y = k_1 - k_2, |k_1|, |k_2| \le M |y|$. A reproducing cone is unflattened (a consequence of the Krein-Smulian theorem \cite{Vu}).

\smallskip
Denote by ${\cal B}(Y)$ is the space of bounded linear operators from $Y$ to $Y$ equipped with the  sup norm. As usual, the spectral properties of an operator $S \in {\cal B}(Y)$ refer to its complexification. Let $r(S)$ be the spectral radius of $S\in {\cal B}(Y)$.

\medskip
An eigenvalue $\lambda \in \sigma(S)$ of $S \in {\cal B}(Y)$ is  {\it basic} if the following properties hold.

\begin{enumerate}
\item[(b-1)] There are associated eigenvectors $\phi > 0 $ of $S$, $\phi^\ast > 0 $ of $S^\ast$.
\item[(b-2)]  Eigenvectors of $S$ (resp. $S^\ast$)  associated with $\lambda$ are multiples of $\phi$ (resp. $\phi^\ast$).
\item[(b-3)]  No $y \in Y$  satisfies $(S - \lambda) y = \phi$.
\end{enumerate}

An eigenvalue $\lambda \in \sigma(S)$ satisfying properties (b-2) and (b-3) above is {\it simple}. A simple eigenvalue  of $S$ is also of $S^\ast$.

An operator $S \in {\cal B}(Y)$ is {\it positive} ($S\ge 0$) with respect to a cone $K$ if $S K  \subset K$;  $S$ is {\it strictly positive} ($S>0$) if $S(K \setminus \{0\}) \subset K \setminus \{0\}$, and
$S $ is  {\it ergodic } if  for any $k \in K\setminus \{0\}$, $S k$ is a quasi-interior element of $K$, $Sk >0$. An operator $S \in {\cal B}(Y)$ is {\it   r-special} if it is compact and ergodic with $r(S)>0$. If $S$ is r-special, $S^\ast \in {\cal B}(Y^\ast)$ also is (recall that $K^\ast$ is also a normal, reproducing cone).

\medskip
The result below is a special case of Theorem \ref{MSS} in the Appendix.

\begin{theo} \label{theo:KRnointerior}
Let $K$  be a normal reproducing cone of a real Banach space $Y$. Then an r-special operator $S \in {\cal B}(Y)$ has a basic eigenvalue $r(S) = r(S^\ast)$.  Eigenvectors of $S$ in $K$  are multiples of $\phi$. For $\mu > r(S)$,   $(\mu I - S)^{-1} K \subset K$. If $\mu \in \sigma(S)\setminus\{ r(S)\}$, $|\mu| < r(S)$.
\end{theo}

We list simple properties of r-special operators frequently used in the text.

\begin{cor} \label{twoprojections}
Let $S \in \BY$ be  an  r-special operator with basic eigenvalue $r(S)$ and associated eigenvectors $\phi>0$ and $\phi^\ast>0$ of $S$ and $S^\ast$. Then the operator $S- r(S)\ I: Y \to Y$ is  a  Fredholm operator  of index $0$, the subspaces
$V = \langle \phi \rangle$ and $W = \Ker \phi^\ast= \Ran (S - r(S) I)$ are invariant under $S$ and $Y = W\oplus V$.

\end{cor}

\begin{proof} A simple consequence of $S$ being r-special.
\qed
\end{proof}

\medskip

The basic eigenvalue $r (S)= r (S^\ast)$ of $S$ and $S^\ast$ and related positive, normalized eigenvectors $\phi$, $\phi^\ast$, vary smoothly with  $S$. We collect some known formulas  (Proposition 79.15  in \cite{Z}, Theorem 4.3 in \cite{Mk}).

\begin{lemma} \label{smoothlambda}
Let $S \in \BY$ be an r-special operator.
There is an open neighborhood ${\cal U}$ of $S$ such that, for $\hat S \in {\cal U}$ the spectral radius $r(\hat S)$ is a simple eigenvalue. The map $\hat S \in {\cal U} \mapsto r(\hat S)$  is real analytic. For appropriate local normalizations, the eigenfunctions $\phi(\hat S)$ and $\phi^\ast(\hat S^\ast)$  are real analytic maps. For a normalization such that $\langle \phi^\ast(\hat S) , \phi(\hat S) \rangle = 1$,
\[ D r(\hat S) \cdot { \overline S} = \langle \phi^\ast(\hat S) , { \overline S} \ \phi(\hat S) \rangle \ . \]

Let $S_\star \in \BY$ be r-special. If $S_\star\ge S$ , then $r(S_\star)\ge r(S)$. Similarly, if $S_\star>S$, then $r(S_\star)>r(S)$. \end{lemma}

\subsection{Spectral theory in $\Hab$} \label{r-H}

Let $T \in \BY$ be an r-special operator with $r(T)=r(T^\ast)=1$ and associated normal eigenvectors $\phi=\phi(T) > 0 $ and $\phi^\ast=\phi(T^\ast) > 0 $ as in Corollary \ref{twoprojections}.

\begin{prop} \label{speciall} For  $H  \in \Hab$, the operator $ H  T\in \BY$ is  r-special, with a basic eigenvalue $r(H T)= r((H T)^\ast)$ bounded away from zero uniformly in $\Hab$.
\end{prop}

\begin{proof} Since $ S$ is r-special and $HT \ge S$ from (r-H), $H T$ is r-special and, by Lemma \ref{smoothlambda}, $r(H T) \ge r(S) >0$ uniformly in $\Hab$. The statements about basic eigenvalues follow from  Theorem \ref{theo:KRnointerior}  and its corollary.
\qed
\end{proof}

\medskip
Endow $\BY$ with the {\it weak operator topology}: for $H_k, H_\infty \in \BY$,
\[ H_k \wconv H_\infty \quad \Leftrightarrow \ \ \forall y \in Y, \ y^\ast \in Y^\ast,
\quad \langle y^\ast ,H_k \ y \rangle \to \langle y^\ast , H_{\infty} \ y \rangle \ . \]

\smallskip
\noindent
Since $K$ and $K^\ast$ are reproducing cones, to show  $ H_k \wconv  H_\infty$, it suffices to prove
\[ \lim \ \langle y^\ast,  H_k \ y \rangle = \langle y^\ast,  H_\infty \ y \rangle \quad \hbox{for} \quad \ y \in K, \ y^\ast\in K^\ast \ . \]

For a separable space $Y$, the weak topology metrizable.

\begin{lemma} \label{Gab} The sets $\Hab$ and $\Haz$ are  convex, compact and sequentially compact in the weak operator topology.
\end{lemma}

\begin{proof} We obtain compactness in the weak operator topology.  Closed balls $\{|| H||\le C\} \subset {\cal B}(Y)$ are compact and sequentially compact in the weak operator topology for a reflexive, separable Banach space $Y$ (\cite{Eldredge}, \cite{Wofsey}). It suffices  to show that $\Hab$  is closed in the strong operator topology: for   $ S \le  H_k  T \le  T, \ \| H_k\|\le R$, if  $ H_k \sconv   H $, the same inequalities hold for $ H$.
\qed
\end{proof}

\medskip
Spectral objects behave well under weak operator convergence.

\begin{lemma} \label{weaklimits}
Let $ H_k \in \Hab$ such that $H_k \wconv H_\infty \in\Hab$.
Consider sequences $\{y_k \in Y, | y_k |=1\}$, $\{g_k\in Y, g_k \to g_\infty \}$, $\{ \lambda_k \in \RR, \ 0 < c < \lambda_k < C\}$ for which
\[ H_k T y_k = \lambda_k y_k + g_k \ . \]
Then, for some (relabeled) subsequences, $y_k \to y_\infty \ne 0$, $\lambda_k \to \lambda_\infty > 0 $ and $H_\infty T y_\infty = \lambda_\infty y_\infty + g_\infty$.
\end{lemma}

\begin{proof} Since $Y$ is separable and reflexive, from the sequential Banach-Alaoglu theorem, there is a subsequence $y_k \rightharpoonup y_\infty \in Y$ and, from the bounds on $\{ \lambda_k\}$, $\lambda_k \to \lambda_\infty \ne 0$. As $T$ is compact, for yet another  subsequence $y_k \rightharpoonup y_\infty \in Y$, $T y_k \to T y_\infty$. Take $\ell \in Y^\ast$ and use $H_k \wconv H_\infty$:
\[ \langle \ell , H_\infty T y_\infty \rangle  = \lim \langle \ell , H_k T y_\infty \rangle = \lim \ \big( \langle \ell , H_k T (y_\infty - y_k) \rangle  +\ \langle \ell , H_k T y_k \rangle \big) \ .\]
As $T(y_k - y_\infty) \to 0$ and the operators $H_k$ are uniformly bounded by (r-H), the first term goes to zero and then
\[ \langle \ell , H_\infty T y_\infty \rangle  =  \lim \ \langle \ell , H_k T y_k \rangle = \lim \ \langle \ell , \lambda_k y_k + g_k\rangle  =  \langle \ell , \lambda_\infty y_\infty + g_\infty \rangle \ ,\]
and thus $H_\infty T y_\infty = \lambda_\infty y_\infty + g_\infty$. If $y_\infty = 0$, then $T y_k \to T y _\infty =0$ and $g_\infty = 0$.
so that, from $ H_k T y_k = \lambda_k y_k + g_k$,
since $| \lambda_k y_k|$ is bounded away from zero and $\{H_k\}$ is bounded, we must have $y_\infty \ne 0$, a contradiction. Thus $y_\infty \ne 0$. Finally, $y_k = (H_k T y_k - g_k)/\lambda_k$ and the right hand side converges: $y_k \to y_\infty$.
\qed
\end{proof}

\medskip
Set $V= \langle \phi(T) \rangle$, $W= \Ker \phi^\ast(T)$ as in Corollary \ref{twoprojections}, so that $Y = W \oplus V$. Define the associated projection $\Pi_W: Y \to W$.

\begin{prop}\label{r-specialMbgen}    For  $a \in (0, 1]$, there  exists ${\bf b} >0$ such that, for  $ H \in \Habb$, the following properties hold.
\begin{enumerate}
\item[(i)] If $1 \in \sigma(HT)$  is an eigenvalue, then $1=r(HT)$.
\item[(ii)] If $(I - HT) w =  c \ \phi(T)$ for $w \in W$, $ c \in \RR$,  then $w = 0$.
\item[(iii)] $HT$ does not have more than one eigenvalue larger than 1.
\item[(iv)] Assume ${\bf b}$ for which items (i)-(iii) hold. Then the operators $I- \Pi_W H T: W \to W$ are uniformly coercive:
\[ \exists \ C > 0, \  \forall \  v \in W,\ \forall \ t\in \RR, \quad |(I- \Pi_W HT) v| \ \ge \ C \ |  v | \ , \]
\end{enumerate}
\end{prop}

\medskip
\begin{proof}  All items will be proved by contradiction.

\noindent (i) Suppose $H_k \in \Habk$, $b_k <1/k$ such that  $1 \in \sigma(H_kT)$ and $r(H_kT)>1$. As $T \in \BY$ is compact, so is $H_kT$, and thus $\sigma(H_kT)\setminus \{0\}$ contains only eigenvalues: there are  eigenvectors (normalized in $Y$) for which $H_k T\psi_k = \psi_k$ and $H_k T\phi_k =r(H_k T) \phi_k$, such that $r(H_k T) > 1$ is a basic eigenvalue, $\phi_k > 0$. From Lemma \ref{Gab},  up to subsequences, $ H_k \wconv  H_\infty \in \Haz$.
 From Lemma \ref{weaklimits}, setting $\lambda_k=1$ and $g_k = 0$,  there exists $\psi_\infty\in Y$  such that (up to subsequences)   $\psi_k \to  \psi_\infty\neq 0$, with  $H_\infty T \psi_\infty =  \psi_\infty$.

\smallskip
From the monotonicity of the spectral radius (Lemma \ref{smoothlambda}), $1 \le r(H_k T) \le (1+ b_k)r( T) \to 1$ and, again by Lemma \ref{weaklimits} with $\lambda_k=r(H_k T)$,  we have
$ \phi_k\to \phi_\infty \ne 0$ in $Y$, with $ H_\infty T  \phi_\infty = \phi_\infty\in X$.

\medskip We show $\phi_\infty \ne\pm \psi_\infty$.
Indeed, there is a sequence of dual  normalized eigenvectors $\phi_k^*\in Y^\ast$, $\phi_k^* > 0 $, $ (H_k T)^* \phi_k^* = \lambda_k \phi_k^*$,
which converges  to $\phi_\infty^* \in Y^\ast$ by an argument as in Lemma \ref{weaklimits}. Also, $(H_\infty T)^* \phi_\infty^* = \phi_\infty^*$,  \ \  $\phi_\infty^* \in K^\ast$, $ \phi_\infty^\ast  \neq 0 $.
From Corollary \ref{twoprojections} (i), $ \langle \psi_k, \phi_k^* \rangle =0$ and by taking limits, $ \langle \psi_\infty, \phi_\infty^* \rangle =0$. If $\phi_\infty =\pm \psi_\infty$, then $ \langle \psi_\infty, \phi_\infty^* \rangle \ne 0$ by positivity. But then $\phi_\infty$ and $\psi_\infty$ are independent eigenvectors, contradicting the simplicity of $r(H_\infty T)$ (Proposition \ref{speciall}).

\medskip
\noindent (ii) Take  $\bf b$ such that item (i) holds.  Write
\[    H_k Tw_k= w_k - c_k \ \phi(T) \ , \quad w_k \in W\setminus \{0\},  \quad b_k = \| B_k \| \to 0 \ . \]
If $c_k=0$, one has $1\in \sigma(H_k T)$ and, from (i), $r(H_k T)=1$. From Proposition \ref{speciall}, we assume $w_k>0$. By positivity,  this contradicts $ w_k \in W= \Ker \phi(T^\ast)$. Thus $c_k \ne 0$: rescale  and consider  $c_k=1$ for all $k$.

\smallskip
If $\ |w_k|$ is unbounded, normalize $\hat w_k = w_k/ |w_k|$. For a subsequence of $\{H_k\}$, $H_k \wconv   H_\infty \in \Haone$ and, from Lemma \ref{weaklimits},  $\hat w_k \to \hat w_\infty \in W \setminus \{0\}$ and
$H_\infty T \hat w_\infty = \hat w_\infty$. Therefore, $1 \in \sigma(H_\infty T )$    and $r(H_\infty T)=1$ by item (i). Again, take $\hat w_\infty >0$, contradicting $ \hat w_\infty \in W= \Ker \phi( T^\ast)$.

\smallskip
Suppose instead that $|w_k|$ is bounded.  For subsequences, $w_k \rightharpoonup w_\infty $ in $Y$,  $ H_k \wconv H_\infty $, and  again from Lemma \ref{weaklimits}, $w_k \to w_\infty \neq 0$ and $w_\infty- H_\infty T w_\infty = \phi(T)$.

 Again by monotonicity,  $1 < r(  H_k T ) \le (1+ b_k)r( T) \to 1$ and $r( H_\infty T)\in [a,1]$. If  $r(H_\infty T)<1$, $w_\infty = (I - H_\infty T)^{-1}  \phi $ and by Theorem  \ref{theo:KRnointerior}, $w_\infty \in K$, contradicting $w_\infty \in W\setminus\{0\} $.  If  $r(H_\infty T) =1$, we must have $\phi \in \Ran (I - H_\infty T) = \ker \phi((H_\infty T)^\ast)$ and again this cannot  happen, as $\phi((H_\infty T)^\ast)>0$.

\medskip
\noindent (iii) Suppose $\lambda, \mu \in \sigma(HT)$ such that $1< \lambda < \mu$ (recall that the top eigenvalue is simple, and the eigenvalues are isolated). Now consider the operator $ H T = HT/ \lambda$, with eigenvalues $1/\lambda < 1< \mu/\lambda$. The operator belongs to an appropriate $\Hab$ (replace $S$ by $S/\lambda$) and thus, if $1 \in \sigma( HT)$, we must have $1 = r( HT)$, by (i), contradicting $1< \mu/\lambda$.

\medskip
\noindent (iv) Suppose $\epsilon_k \to 0$ and $v_k \in W$, $t_k\in \RR$ for which
$|(I- \Pi_WHT) v_k| \ < \ \epsilon_k \ |  v_k | $.
Thus $v_k \ne 0$ and one may normalize, $z_k = v_k/|v_k|$, so that $z_k - \Pi_W  H_k T z_k \to 0$. By Proposition \ref{speciall}, for a subsequence, $H_k \wconv G_\infty \in \Hab$. As in Lemma \ref{weaklimits}, for a subsequence $z_k \to z_\infty\in W$,
$ z_\infty - \Pi_W  G_\infty  T  z_\infty =0$ so that $ z_\infty - G_\infty T z_\infty = c \phi(T)$.

By item (ii), since $z_\infty \in W$, we must have $c = 0$ and then $ z_\infty =0$. On the other hand, since $ z_k \to  z_\infty=0$ and the $H_k$'s are uniformly bounded,
  $$|H_k T z_k| \to 0, \quad {\rm and}\quad | \Pi_W H_k T z_k |\to 0, $$
yielding a contradiction:  $z_k - \Pi_W  H_k T z_k \to 0$ and  $|z_k|=1$.
\qed
\end{proof}

\subsection{Hypothesis (r-H): adapted coordinates} \label{r-HH}

Let $T \in \BY$ as in Theorem \ref{theo:F}, basic eigenvalue $r(T)$ and associated eigenfunctions $\phi= \phi(T) \in Y$, $\phi^\ast = \phi(T^\ast)\in Y^\ast$. From Corollary \ref{twoprojections}, the decomposition $Y = W \oplus V$ for $V = \langle \phi\rangle$ and $W = \Ker \phi^*$ induces projections $\Pi_W: Y \to W$ and $\Pi_V:Y \to V$.
The {\it horizontal hyperplane} at height $t$ is  $W^t = W + t \phi$. Define the projected restrictions of $F$,
\[ F^t : W \to W, \quad F^t(w) = w - \Pi_W  P(T(w + t \phi))  . \]

The next proposition finesses some elliptic estimates employed in Theorem \ref{BNV}.

\begin{prop} \label{coercivelip} Assume the hypotheses of Theorem \ref{theo:F} and (r-H) for ${\bf b}$ such that Proposition \ref{r-specialMbgen} holds. Then $F^t: W \to W$ are uniformly Lipschitz coercive:
\[ \exists \ C > 0:\  \  \forall \ w ,\  v \in W,\ \forall \ t\in \RR, \quad | F^t(w) - F^t(v)| \ \ge \ C \ | w - v | \ , \]
and so are the maps $\Psi: W \oplus \RR \to W \oplus \RR, (w, t) \mapsto (z,t) = (F^t(w), t)$.
\end{prop}

\begin{proof}
Suppose by contradiction $\epsilon_k \to 0$ and $w_k, v_k \in W$,   $t_k\in \RR$ for which
\[ | \Ftkr (w_k) - \Ftkr(v_k)| \ < \ \epsilon_k \ | w_k - v_k | . \]
For the  linearization $G_k = G_k(T(w_k + t_k \phi), T(v_k + t_k \phi)) \in \Hab$,
\[ | \Ftkr (w_k) - \Ftkr(v_k)|  = | (w_k - v_k) - \Pi_W \big(P(T(w_k + t_k \phi))  - P(T(v_k + t_k \phi))\big) | \]
\[ = | (w_k - v_k) - \Pi_W  G_k  T(w_k - v_k) |  > C \ | w_k - v_k | . \]
for the uniform Lipschitz bound $C$ in Proposition \ref{r-specialMbgen}(iv).

As for $\Psi$, use the equivalent norm $|||u||| = |||w + t \phi||| = |w| + |t|$.
\qed
\end{proof}
\medskip
For the rest of Section \ref{BNV}, $b=\bf b$ is taken as in Proposition \ref{r-specialMbgen}. Indeed, this is the required $b>0$ in Theorem \ref{theo:F}.

\medskip
We recall two well known facts. Let $Z$ be a real Banach space and $U \subset Z$ be an open subset. A continuous function $f: U \to Z$ is {\it compact} if the image of bounded sets is relatively compact.

\begin{theo*} [Banach-Mazur, (5.1.4) in \cite{Berger}] A continuous map $H: Z \to Z$ is a homeomorphism if and only if it is proper and a local homeomorphism.
\end{theo*}

\begin{theo*}[Schauder Invariance of Domain, \cite{Schauder}] Let $H:U \to Z$ be a continuous injective function of the form $H(x) = x - Q(x)$ for a compact map $Q:U \to Z$. Then $H(U) \subset B$ is an open set.
\end{theo*}

\medskip
We introduce {\it adapted coordinates} (\cite{BP}, \cite{CT}).
\begin{prop}  \label{adaptedcoordinatesFr} Assume the hypotheses of Proposition \ref{coercivelip}. For $V  \simeq \RR$, the map
\[ \Psi: Y = W \oplus \RR \to Y = W \oplus \RR, (w, t) \mapsto (z,t) = (F^t(w), t) \]
is a homeomorphism. For $F^a(z,t) = F \circ \Psi^{-1}$, the following diagram commutes.
\[
\begin{array}{ccl}
{(w, t)}&
\stackrel{{\scriptstyle F}}{\longrightarrow}&
{(F^t(w), \Pi_V F(w+t\phi)) } \\
\\
   {\scriptstyle  \Psi^{-1}} \nwarrow & &
\nearrow{\scriptstyle F^a(z,t) = (z , h^a(z,t)) = (z,\Pi_V F(\Psi^{-1}(z,t)))}\\
    &{(z,t) } &  \\
  \end{array}
\]
\end{prop}

\medskip
\begin{proof}  Write $y = w + t \phi = \Pi_W y + t \phi$. We apply the Banach-Mazur theorem to
\[ \Psi(w,t) =  F^t(w) + t \phi = (y -t\phi) -  \Pi_W P(Ty) + t \phi \equiv y - Q(y). \]
The properness of $\Psi$ follows from Proposition \ref{coercivelip}. Indeed, given sequences such that $\Psi (w_k,t_k) = (z_k, t_k) \to (z_{\infty},t_{\infty})$,  use the equivalent norm $|||y||| = |w| + |t|$,
\[ | w_n - w_m| + | t_n - t_m| = | \Psi (w_n,t_n) - \Psi (w_m,t_m) | \ge C (| w_n - w_m| + | t_n - t_m| )\ , \]
to conclude that the  sequences $\{(z_k, t_k)\}$ and $\{(w_k,t_k)\}$ are both Cauchy sequences, necessarily convergent. The same estimate implies the injectivity of $\Psi$. We verify local surjectivity: the continuity of the inverse again would follow from Proposition \ref{coercivelip}. By the Banach-Mazur theorem, it suffices to show that $\Psi$ is an open map, which implies local surjectivity: we use Schauder's theorem.

We show that, around each $(w,t) \in W \oplus V$, there is a closed ball $D \subset Y$  whose image $Q(D) \subset Y$ is relatively compact for $Q(y) =  \Pi_W  P(T y)$. Since $T$ is compact, and $P$ continuous (Lipschitz), $P\circ T$ is a compact  map,
and the  linear, bounded operators $\Pi_W$ preserves compactness.

Thus $\Psi$ is a genuine change of variables. The diagram now is immediate.
\qed
\end{proof}

\medskip
Thus, for each $t \in \RR$ and $z \in W$, there is a unique $w(z) \in W$ for which $\Ftl(w(z)) = z$. Said differently, the image under $F$ of a horizontal hyperplane $W^t$ intercepts the vertical line $L_z = \{ z + h \phi\}$ at a single point: each hyperplane $W^t$ contains a unique preimage of $L_z$.  We denote by $\Lambda_z = F^{-1}(L_z)$ the {\it fiber} associated with $z$. A fiber  contains a unique point on each $W^t$ and is parameterized by $\RR$: $\Lambda_z(t) = \{ y(t) = w(z,t) + t \phi, t \in \RR\}$.

The map $\Psi^{-1}$ takes each $W^t$ to itself and vertical lines are taken to fibers.
The function $t \mapsto h(z, t )=\Pi_V F(y(t)) = \langle \phi(T^\ast), F(y(t)) \rangle$ takes the point at height $t$ in the fiber $\Lambda_z$ to the point at height $h(z,t)$ of the vertical line $L_z$. Equivalently, $h^a(z, \cdot )$ takes $(z,t)$ to a point at height $h(z,t)$ of the vertical line $L_z$.

\medskip
\noindent{\bf Remark} For a given $z \in Y$, the solutions of $F(y) = z$ lie in the fiber $\Lambda_z$ which, in principle, may be searched by a (one dimensional) continuation method (\cite{S}). Numerical algorithms have been written exploiting this special feature (\cite{CT}).

\begin{prop} \label{injectivity} Under the hypotheses of Theorem \ref{theo:F} (1), $F:Y \to Y$ is injective.
\end{prop}

\begin{proof} For $y_1,y_2 \in Y$ for which $F(y_1) - F(y_2) = 0$, use linearizations,
\[ F(y_1) - F (y_2) = (y_1 - y_2) -  G(Ty_1,Ty_2) \  T(y_1-y_2) = 0 \ .\]
From Proposition \ref{speciall}, $G(Ty_1,Ty_2) \in \Hab$, so that $1 \in \sigma (G(Ty_1,Ty_2)T)$ and thus $r(y_1,y_2)=1$, by Proposition \ref{r-specialMbgen}, contradicting the hypothesis.
\qed
\end{proof}

\medskip
\noindent{\bf Proof of Theorem \ref{theo:F}(1):} From Proposition \ref{injectivity}, $F$ is  injective. Local surjectivity is again a consequence of Schauder's theorem, as in the proof of Theorem \ref{adaptedcoordinatesFr}. We consider the continuity of the (local) inverse.
In the notation of Proposition \ref{adaptedcoordinatesFr}, it suffices to show that  $F^a: W \oplus \RR \to W \oplus \RR, (z,t) \mapsto (z, h^a(z,t))$ is a local homeomorphism. The exercise in real analysis is left to the reader.
\qed

\subsection{Hypothesis (r-Conv): simple maps} \label{proofF}

As in Theorem \ref{theo:F}(2), suppose (r-H), (r-Conv) hold and ${\bf b}$ is taken such that Proposition \ref{r-specialMbgen} is true. Set $r(G(Ty_1,Ty_2)T)\equiv r(y_1,y_2)$ and $r(G(Ty,0)T)\equiv r(y)$.

\begin{prop} \label{NT} If $F(y_1) = F(y_2)$, then trichotomy holds: either $y_1 > y_2$ or $y_1 = y_2$ or $y_1 <y_2$. For  $y \in Y$, the equation $F(y_1) = y$ has at most two solutions. \end{prop}

\begin{proof} As in Proposition  \ref{injectivity}, if $F(y_1) - F(y_2) = 0$,  $G(Ty_1,Ty_2) \in \Hab$ and  $r(y_1,y_2)=1$. From Theorem \ref{theo:KRnointerior}, for a choice of sign, $\pm (y_1-y_2)$ is a quasi-interior eigenvector of $G(Ty_1,Ty_2)T \in \BY$: trichotomy holds.

Suppose  that  $z \in Y$ has three preimages, $z=F(y_1) = F(y_2)= F(y_3)$ for which $y_1 < y_2 < y_3$. As before,
$r(y_3,y_2) = r(G(Ty_2,Ty_1)T) = 1$. But from (r-Conv), since $ T y_1 <  T y_2 < T y_3$ (as $ T$ is r-special), we must have
$ G(Ty_3,Ty_2) >  G(Ty_2,Ty_1)$ and thus $r(y_3,y_2) > r(y_2,y_1)$, from the strict monotonicity of the spectral radius (Lemma \ref{smoothlambda}), a contradiction.
\qed
\end{proof}

\medskip
If (r-Conv) is not satisfied, the equation $F(y) = z$ for a fixed $z \in Y$ may have a number of solutions, but their graphs still sit one on top of the other.

\medskip
\noindent{\bf Proof of Theorem \ref{theo:F} (2):} The homeomorphism $\Psi: W \oplus V \to W \oplus V$ in Proposition \ref{adaptedcoordinatesFr} preserves the vertical component: it suffices to prove that $F^a: W \oplus V \to W \oplus V$ is a simple map. This is immediate from Proposition \ref{NT}: notice that all preimages of a point $z \in Y$ necessarily belong to $F^a(L_z)$, where $L_z \subset Y$ is the vertical line  through $z$. \qed

\medskip
A finer geometric description of the map $F:Y \to Y$ follows from the asymptotic behavior of the images of its fibers, as in Theorem \ref{theo:BNV2} below.

\subsection{Nemitskii maps: proof of Theorem \ref{theo:BNV}} \label{Nemitskii}

At this point, standard techniques obtain from Theorem \ref{theo:F} (2) an extension of Theorem \ref{theo:BNV}. We resume the proof of Theorem \ref{theo:F} in Section \ref{properness}.

Let $Y$ be a real, separable, reflexive Banach space {\it of functions} containing the constants, for which $K$, the cone of pointwise non-negative functions, is normal and generating.  We assume that multiplication $M_q$ by bounded functions $q: \RR\to \RR$ be bounded, $\|M_q\| \le \sup |q|$. As in Theorem \ref{theo:F}, $T  \in \BY$  is an r-special operator for which $r(T)=1, T \phi = \phi >0, T^\ast \phi^\ast = \phi^\ast >0, \langle \phi^\ast,  \phi \rangle = 1$.

\begin{theo} \label{theo:BNV2}  Let $P(u)= f(u)$ for a strictly convex function $f: \RR \to \RR$ such that
\[ 0 < a = \inf_{r \ne s} q (r,s) < 1 < 1 + b = \sup_{r \ne s} q (r,s) \ . \]
Then, for some $ b > 0$, $F = I - PT: Y \to Y$ folds downwards with respect to $\phi$.
\end{theo}

Here, the Newton quotient $q(r,s)$ is $q(r,s) = (q(r) - q(s)) / (r-s)$ for $r \ne s$.

\begin{proof}
We check the hypotheses of Theorem \ref{theo:F} (2). As in \cite{AP}, use linearizations  $G(u,v)=M_{q(u,v)}$, where $M_{q(u,v)}$ is the multiplication operator  by the bounded function
$q(u, v) (x)\ = \
q(u(x),v(x))$ for $u(x) \neq v(x)$ and $q(u,v)(x)= a$ otherwise.

\smallskip
Clearly $ P, G(u,v): Y \to Y$ are well defined and continuous, as $f$ is Lipschitz. Moreover, $M_{q(u,v)} = G(u,v) \in \Hab$ for $ S = a  T$, $R = 1+ b$. Indeed,
\[ q(u,v) = \alpha + \beta = \min \{q(u,v), 1\} + (q(u,v) - \min \{q(u,v), 1\}) ,  M_{q(u,v)} = M_\alpha + M_\beta \ , \]
and then $ S= a  T \le M_\alpha  T \le  T$,
$\| M_\beta \| \le b$. This settles (r-H).

\medskip
For (r-Conv), consider pointwise estimates for $x \in \Omega$, the domain of the functions in $Y$: as $f$ is strictly convex,
for $y_1, y_2, z_1, z_2 \in Y$ as in (m-Conv),
\[ \frac{f(y_1(x)) - f(z_1(x))}{y_1(x)- z_1(x)} > \frac{f(y_2(x)) - f(z_2(x))}{y_2(x)- z_2(x)} \ . \]
Thus $ G(y_1,z_1)  -  G(y_2,z_2) \in \BY$ is  strictly positive.
Now use Theorem \ref{theo:F} (2) to conclude that $F$ is a simple map.

 From the hypotheses on $f$, there are lines
$\ell_\pm(t) = \alpha_\pm t + \beta_\pm$ such that $f(t) > \ell_\pm(t)$ with $\alpha_- < 1 = r( T) < \alpha_+$ (\cite{STZ}, for example).

We compute the height of the images of the fiber $y(t) = w(t) + t \phi, w(t) \in W$:
\[ h(F(y(t))) =\langle \phi^\ast, F(y(t))\rangle = \langle\phi^\ast, w(t) + t \phi - P( T y(t))\rangle = t - \langle \phi^\ast,f(  T y(t))\rangle \ . \]
\[ < \ t -  \langle\phi^\ast,\alpha_+ T (w(t) + t \phi) + \beta _+\rangle = (1 - \alpha_+) t - \langle\phi^\ast, \beta_+ \rangle , \]
(recall $\beta_+ \in Y$). As $t \to \infty$, we have $h(F(y(t))) \to -\infty$ since $1 - \alpha_+<0$. A similar estimate using the line $\ell_-$ obtains $h(F(y(t))) \to -\infty$ also for $t \to -\infty$.

By definition, the image $F(y(t))$ of the fiber $y(t)$  lies in a vertical line of $Y$. As $t \to \pm \infty$,  $h(F(y(t))) \to - \infty$: $F$ folds downwards.
\qed
\end{proof}

\medskip
The set $\Hab$ is not necessary in this case: the implicit set in the arguments  is
\[ \Qab = \{ M_q\in{\cal B}(Y): \ q  \in L^\infty(\Omega) \ , \  a \le q(x) \le 1+ b \ \ {\rm a.e.} \} \ , \]
which is  convex and sequentially compact with respect to the weak operator topology of ${\cal B}(Y)$ from the Banach-Alaoglu Theorem. Clearly, $\Qab \subset \Hab$.

\medskip
\noindent{\bf Proof of Theorem \ref{theo:BNV}}

The space $Y= L^n(\Omega)$ clearly satisfies the hypotheses of Theorem \ref{theo:BNV2}, as does $K \subset Y$, the cone of nonnegative functions (notice that nonnegative functions in $X=W^{2,n}(\Omega) \cap W^{1,n}_0 (\Omega)$ do not form a normal cone). Without loss, $\lambda_m= \lambda_m(\Lb) = 0$ (simply replace $\Lb$ by $\Lb - \lambda_m$). Denote by $\CC^+ = \{ z \in \CC , \Re z > 0 \}$ the open right half-plane: by hypothesis, $\sigma(\Lb) \subset \CC^+ \cup \{0\}$.

Recall $X = W^{2,n} \cap W^{1,n}_0 \subset C^0 $. The operator $\Lb +\gamma I: X \to Y$ is invertible for $\gamma>0$. Denote the compact inclusion of $X$ in $Y$ by $\iota: X \to Y$.  Given the conformal mapping $z \mapsto \Gamma_\gamma(z) = \gamma /(z + \gamma)$, $\gamma >0$, define the operator $T = \iota \circ \gamma (\Lb +\gamma I)^{-1} \in \BY$. As $\Gamma_\gamma$ takes $\CC^+$ to the open disk $D= \{ |z - 1/2| < 1/2\}$, $\sigma(T) \subset D \cup \{1\}$ (recall the Dunford-Schwartz functional calculus for closed operators $L : \Dom (L) \to Y$ \cite{DS}, p. 599).
Also, $\phi_m > 0 $ is an eigenvector of both $\Lb$ and $T$ for eigenvalues $0$ and  $r(T)=1$ respectively.

The hypotheses of Theorem \ref{theo:F} for $T$ are satisfied: from well known properties of the operator $\Lb: X \to Y$ (\cite{BNV}, \cite{B}), $T \in \BY$ is r-special. Also, $r( T)=1, T \phi_m = \phi_m \in Y, \phi_m>0$ and similar hypotheses hold for $\phi^\ast_m$, since $T$ is r-special.

\medskip
The map $\Pb: Y \to Y, \ \Pb(u) = \fb(u)$ is associated with  a strictly convex function $\fb: \RR \to \RR$ for which
$ \ab = \inf q_{\fb}(x,y) <  0 <  \bb = \sup q_{\fb} (x,y)$.
Set $P = I+ \Pb/\gamma$, so that $P(y) = f(y)$ for $f = 1 + \fb/\gamma$ and
\[ 0 < a =  1 + \ab/\gamma < r(\tilde T)= 1 < 1+ b = 1 + \bb/\gamma \ . \]

As in the proof of Theorem \ref{theo:BNV2}, $P$ satisfies (r-H) and (r-Conv). Now take $\gamma$ so large that $b >0$ complies with the hypothesis of Theorem \ref{theo:BNV2}. Thus $F= I - PT:Y \to Y$ folds downwards, and the same must happen to $\Fb= \Lb - \Pb: X \to Y$: Theorem \ref{theo:BNV} holds for the appropriate $\bb = \gamma b$.
\qed

\medskip
We consider extensions of Theorem \ref{theo:BNV}, by replacing $\Lb$ with an appropriate $g(\Lb)$  which still induces an r-special operators $T$ through the conformal mapping $\Gamma$. We assume $\lambda_m=0$.

As a simple example, for $\Lb$ as above, powers $\Lb^k$, for some $k \in \NN$ still yield r-special operators, provided $\sigma(\Lb^k)$ lies in the closed half-plane.
Clearly, this is the case if $\sigma(\Lb) \subset [0, \infty]$. No other eigenvalue of $\Lb^k$ is sent to zero and thus $T = \Gamma_\gamma(\Lb^k)$ still has a {\it simple} spectral radius 1. Compactness and preservation of the cone $K$ are immediate.

One might consider more complicated functions $g$, at the cost of additional hypothesis (a decay rate of the resolvent of $L$ would be natural). For example, sectorial properties of  $\sigma(\Lb)$ might lead to fractional powers for which the associated $T$ is still r-special. For a related situation, see Proposition \ref{prop:m-special}.

\subsection{Consequences of properness: Theorem \ref{theo:F} (3)} \label{section:cons}

The function below is already given in adapted coordinates,
\[ (x,y) \mapsto (x, (1- xy)y )  \ . \]
Vertical lines in the right (resp. left) half-plane fold downwards (resp. upward), the vertical axis stays fixed. This map is not proper. Under properness, the  height functions of the image of all vertical lines  have the same asymptotic behavior, as shown in the proposition below. (Proposition 10 of \cite{STZ}, in the current notation).

\begin{prop} \label{prop:STZ} Let $Z$ be a real Banach space. A proper map $A : Z \oplus \RR \to Z \oplus \RR$ of the form $A(z,t) = \big(z, \alpha (z,t)\big)$ for which no point has more than two preimages is either a homeomorphism or folds vertically.
\end{prop}

\medskip
\noindent{\bf Proof of Theorem \ref{theo:F} (3):} Combine Theorem \ref{theo:F} (2) and Proposition \ref{prop:STZ}. \qed

\medskip
For a fiber $\Lambda= \{u(t), t \in \RR\}$, asymptotic information of the height $h(t)$ of its image $F(u(t)) = z + h(t) \phi$ determines if $F$ is a homeomorphism or a fold, as well as the type of the fold. By Proposition \ref{prop:STZ}, if $F$ is proper, such limits are independent of the fiber. One approach to the asymptotic limits has been used at the end of the proof of Theorem \ref{theo:BNV2}: the hypothesis on $f$ is so stringent that the images of fibers {\it and} vertical lines in the domain have the same asymptotic behavior. We present an alternative approach in Section \ref{properness}.

\subsection{Hypotheses (r-Hs), (r-Convs) and (r-Crit): smoothness} \label{properness}

When $F$ is a $C^1$ map we use different tools.
From Proposition \ref{prop:STZ}, the proof of Theorem \ref{theo:F} (4) is complete once we establish that $F$ is proper (Proposition \ref{properl}) and $F$ folds downwards (Proposition \ref{index} (3)).

\medskip
Throughout this section, we assume the hypotheses of Theorem \ref{theo:F}(4): (r-Hs), (r-Convs) and (r-Crit) and the usual ${\bf b} >0$ for which Proposition \ref{r-specialMbgen} holds. For $F(y) = y - P(Ty): Y \to Y$, we have $DF(y) z = z - DP (Ty) T $. From (r-Crit), there is $y_c \in Y$  for which $DF(y_c)$ is not invertible. Without loss,
\[ y_c=0 \ \ \quad F(0)=0 \] (compose $F$ with translations in domain and counterdomain --- hypotheses (r-Hs) and (r-Convs) are invariant under these operations).

We take the linearization $G(y,z) \in {\cal B}(Y)$ to be
\[
P(y)- P(z)  = \int_0^1 D P(z+s ( y-z)) \ ds \  (y-z) =:  G(y,z)(y-z) \ ,
\]
so that $G(y,y)= J(y) = DP(y)$. Set $G(Ty)=G(Ty,0)$. As $F(0)=0$ we have $P(Ty)= G(Ty)Ty$ with $G(Ty)T \in \BY$. Also, $F(y) = y - G(Ty)Ty$.

\begin{prop} \label{prop:convs} For $y,z \in Y$, $J(y), G(y,z) \in \Hab$. The maps \[ y,z \mapsto G(y,z) \in \BY \quad \hbox{and} \quad y,z \mapsto r(G(y,z)T) \in \RR \] are continuous. The linearizations $G(y,z)$ satisfy (r-H) and (r-Conv).
\end{prop}

\begin{proof} From (r-Hs), $DP(y), DP(z) \in \Hab$. From Lemma \ref{Gab}, convex combinations of Jacobians, as well as integrals along a segment, also belong to $\Hab$ and satisfy (r-H) and (r-Conv). The map $y,z, \mapsto G(y,z) \in \BY$ is continuous as $P:Y\to Y$ is $C^1$. Lemma \ref{smoothlambda} implies the continuity of $r(G(y,z)T)$.
\qed
\end{proof}

\medskip

The {\it critical set} $\Cr$ of $F$ consists of the points in $Y$ for which $DF$ is not invertible. From  Propositions \ref{speciall} and \ref{r-specialMbgen},
\[ \Cr  = \{ y \in Y \ | \ 0 \in \sigma(DF(y))\} = \{ y \in Y \ | \  r(y) = r(J(T y)T) =  1 \} \ . \]

Our next step is Proposition \ref{(M-V)}: critical points are abundant. We obtain (exactly) one for each line  $\{y + s \psi, t \in \RR, \psi > 0 \}$. Set $B_\delta^Y  = \{ z \in Y, | z | \le \delta\}$.

\begin{lemma} \label{Hahn-Banach} Let $y, z, \psi \in Y, \psi >0, \delta > 0$. Then there are $s^\ast \in \RR$, $k \in K$ and $z_\delta \in B_\delta^Y$ such that $y + s^\ast \psi = z + z_\delta + k$.
\end{lemma}

Let $z \in L^2(0,1)$ not bounded from below, $y=0, \psi \equiv 1$. Then no constant $s$ gives $0 < s \psi = z+k $: the small perturbation $z_\delta$ is needed.

\medskip
\begin{proof}
If one cannot obtain $- z  + y + s^\ast \psi =  k +   z_\delta$, the  line $ L= \{ -z +y + s \psi, s \in \RR\}$ does not intercept $K_\delta$, the  (closed) convex span of $K$ and  $ B_\delta^Y$,  a set having $\psi \in K$ in its interior. By the Hahn-Banach theorem, $L $ and $K_\delta$ are separated by $H = \{ z \in Y,\  \langle \ell ,z\rangle  = c\}$, the level $c \in \RR$ of a functional $\ell \in Y^\ast \setminus\{0\}$: $\ell|_{{L }} \le c \le \ell|_{{K_\delta}}$.

\smallskip
We first show that $\ell \in K^\ast\setminus \{0\}$.  Indeed, if  there is $k\in K$ such that $\langle \ell, k\rangle<0$, then  for large  positive $t$ (so that $t k \in K$),  $\langle \ell, t\, k  \rangle<c $, a contradiction. Also, since $\ell(y + s \psi)= \ell(y) + s \ell(\psi) \le c$ for all $s \in \RR$, we must have $\ell(\psi)=0$. But, as $\ell \in K^\ast \setminus\{0\}$ and $\psi >0$, we must have $\ell(\psi) >0$.
\qed
\end{proof}

\medskip
As usual, $T \in \BY$ is  r-special, $r(T)=1$. Set $r(y) = r(G(Ty)T)$.

\begin{prop} \label{(M-V)} Let $y \in Y$, $\psi>0$.
\begin{enumerate}
\item[(i)] The map $s \in \RR \to r(y+s \psi) = r(G(T(y+ s\psi),0) T)$  is strictly increasing. \\ For some $s_0 \in \RR$, $r(y+ s_0 \psi) = 1$.
\item[(ii)]
For  all increasing  sequences $s_k \to \infty$, the operators $\{   G(Ts_k \psi)\}\subset \Habb$ have a common weak limit $   G(Ts_k \psi) \wconv   H_\infty  \in \Habb$. An analogous result holds for decreasing sequences $s_k\to -\infty$.
\end{enumerate}
\end{prop}

\begin{proof} (i): By Proposition \ref{speciall}, $G(T(y+ s\psi))T \in \BY$ is  r-special and the eigenvalue $r(y+ s\psi)$ is basic. The map $s \mapsto r(y+ s\psi)$
is a continuous, strictly increasing function, from Lemma \ref{smoothlambda}. We search for $s_+$ for which $r(y+ s_+\psi)>1$.

From (r-Crit),  $0 \in \Cr$, so that $ r(0)= r(G(0,0)T)= 1$. By monotonicity (Lemma \ref{smoothlambda}), for $z_+ = p \psi, p>0$, we have $ r(z_+)=
 r(G(Tz_+)T) > 1$. By continuity, there is $\delta > 0$ for which
$r(z_+ + z_\delta) > 1$ for $z_\delta \in B_\delta^Y$.

Apply Lemma \ref{Hahn-Banach} for $z = z_+$: we obtain $s^\ast \in \RR$, $k \in K$, and $z_{\delta} \in B_{\delta}^Y$ such that
$y + s^\ast   \psi =   z_+ + k +  z_{\delta}$.
We compare the spectral radius at different points:
\[ r( y + s^\ast \psi) =  r(z_+ + k  +  z_\delta) \ge r(z_+   +  z_\delta)> 1 . \]
Set $s_+ = s^\ast$.
Simple modifications of Lemma \ref{Hahn-Banach} obtain $s_-$ for which $r(y+ s_-\psi)<1$. By continuity and monotonicity, there is a single $s_0$ such that $r(y+ s_0\psi)=1$.

\smallskip (ii):
From  (r-Convs), as $  \psi > 0$,  the linearizations $  G(T s\psi)$ increase with $s$. Increasing sequences $\{ s^1_k\}$ and $\{s^2_k\}$ yield limits $  H_1$ and $  H_2$ by Lemma \ref{Gab}. For $y^\ast \in K^\ast, \ z \in K,$ the sequences
$\tau_k^1 = \langle y^\ast,   G(T s^1_k \psi) \ z \rangle$ and $\tau_k^2 = \langle y^\ast,    G(Ts^2_k \psi ) \ z \rangle$ increase and interlace: for every $N \in \NN$, there are $m,n > N$ such that $\tau_m^1 > \tau_n^2$ and other $m,n > N$ with $\tau_m^1 < \tau_n^2$. Thus $  H_1=  H_2:=  H_\infty$.
\qed
\end{proof}

\begin{cor} \label{cor:abundantcritical} Every vertical line $\{ w + t \phi(T), t \in \RR\}$ in the domain $Y$ of $F$ contains a critical point of $F$. The same holds for lines $\{t \psi, \psi>0, t \in \RR\}$.
\end{cor}

\begin{prop}\label{properl} $F: Y \to Y$ is proper.

\end{prop}

The argument follows in spirit the proof of the properness of $F$ in \cite{AP}.

\begin{proof}
Let $\{y_k \}\subset Y$ be a sequence  for which $F(y_k) =g_k \to g_\infty \in Y$. Suppose first that $\{   y_k\}$ is bounded in $Y$. Since the linearizations $  G(y,z),\  G(y)$ are uniformly bounded and $P(0)=0$,
\[ | P(T y_k)| =| G(T y_k) y_k|  \le C|y_k|  \]
is also bounded in $Y$, and the sequence $y_k=  P(T y_k)+g_k$ is bounded. Therefore, for some subsequence, $y_k  \rightharpoonup  y_\infty$. By the compactness of $T$ (and the fact that $P$ is Lipschitz), again up to a subsequence,  $  T y_k \to T y_\infty $ and thus  $P(T y_k)\to P(Ty_\infty)$.  Therefore,
$y_k\to y_\infty =P(Ty_\infty) + g_\infty$.

Consider now $ \{y_k\}$ unbounded in $Y$, for strictly increasing  $|y_k|  \to \infty$.
Normalize $\hat y_k =  y_k/|y_k|  $ and use the linearization $G_k=  G(T y_k)$:
\[\hat y_k -  P(T y_k)/|   y_k|  =  \hat y_k -  G_k \hat y_k = g_k/|y_k|  \to 0 .\]
As usual, since $G_k\in \Hab$, up to subsequences, $G_k \wconv G_\infty$  for  some $     G_\infty \in \Hab$, and from  Lemma \ref{weaklimits},
\[ \hat y_k \to \hat y_\infty\ne 0, \ \hbox{and} \  G_\infty   T \hat y_\infty = \hat y_\infty. \]
By Propositions \ref{speciall} and \ref{r-specialMbgen},
$r(G_\infty T)=1$ is a basic eigenvalue and  either $   \hat y_\infty$  or $-  \hat y_\infty $ is a quasi-interior point.
   Say $  \hat y_\infty > 0$.

Consider now the projections $y_k^p = |y_k|  \hat y_\infty$ along the ray
$\{ t \hat y_\infty, t >0 \}$.
Write the equation $ y_k -   P(T y_k) = g_k$ as

\begin{equation}\label{ast}
  y_k^p -    P (T y_k^p) =     P(T y_k) -      P(T y_k^p)+  g_k+  y_k^p-  y_k \ .\quad \quad
\end{equation}

The right hand side is $o(|  y_k|  )$. This is clear for the bounded sequence $\{   g_k\} $ and for $  (y_k^p-y_k)=|  y_k|(\hat y_k- \hat y_\infty)$, since $\hat y_k \to \hat y_\infty \in Y$.
Moreover,
\[ | P (T y_k) - P (T y_k^p)| \le C|y_k-   y_k^p|  =C|   y_k|  |\hat y_k- \hat y_\infty|  =o(|   y_k|  ). \]
Set $P (T y_k^p) = G_k^p \ T y_k^p$. By Lemma  \ref{Gab}, for a subsequence, $G_k^p \wconv G_\infty^p \in \Hab$.

Divide equation (\ref{ast}) by $|  y_k| $ to obtain
\[ \lim \ \big(y_k^p/|y_k|  - P(T y_k^p)/|  y_k|  \big) =  \hat y_\infty -     G_\infty^p \ T \hat y_\infty  =   0 , \]
and, as $\hat y_\infty > 0$,
$1 \in \sigma (G_\infty^p T )$ and, by Theorem \ref{theo:KRnointerior}, $r(G_\infty^p T ) = 1$. This contradicts Proposition \ref{(M-V)}: the points $y_k^p$ belong to a common ray $\{s \hat u_\infty = s \psi, s \in \RR^+\}$. Combining items (i) and (ii),  $r(G_k^p T)$
increases and takes a value larger than one: the limit cannot be one.

Minor adjustments handle also the case $  \hat y_\infty < 0$.
\qed
\end{proof}

\medskip
When $F:Y \to Y$ is a $C^1$ map, the adapted coordinates in Proposition \ref{adaptedcoordinatesFr} are obtained from $C^1$ changes of variables $\Psi: Y \to Y$. The inverses of vertical lines $L_z= \{ z + h \phi, h \in \RR\}$  under $F$, the fibers
\[ \Lambda_z = \{ y(t) = w(z,t) + t\phi, w(z,t) \in W, \ t \in \RR\} \ , \] have $C^1$ parameterizations. Fix $z \in W$, set $F(y(t)) = z + h(t) \phi$, where again $h(\cdot): \RR \to \RR$ is a $C^1$ map. Let $\lambda(y(t))= 1 - r(DP(Ty(t)T))$.

\begin{lemma} \label{handr} For a continuous strictly positive $\alpha: \RR \to \RR^+$, $h'(t) = \lambda(y(t)) \alpha(t)$.
\end{lemma}

Part of this lemma is Proposition 2.6 in \cite{CTZ2}.

\begin{proof} Differentiate $F(y(t)) = z + h(t) \phi$ along a fiber $y(t)$ to obtain
\[ DF(y(t)) y'(t) = h'(t) \phi \ . \]
Since $\Pi_W DF(y(t)): W \to W$ is always invertible, at a critical point $y(t_c)$, we must have $h'(t_c)=0$ and $\phi \notin   DF(y(t))$. Thus $y'(t_c)$ is a generator of $\Ker DF(y(t_c))$, and either $y'(t_c)>0$ or $-y'(t_c)>0$.  As $y'(t_c) = w'(t_c) + \phi$ and $W = \Ker \phi^\ast$,
\[ \langle \phi^\ast, y'(t_c) \rangle  =  \langle \phi^\ast, w'(t_c) + \phi \rangle
 =  \langle \phi^\ast,\phi \rangle > 0 \]
and $y'(t_c)>0$. Conversely, if $h'(\bar t)=0$, as $y' (\bar t)= w'(\bar t) + \phi \ne 0$, we must have that $DF(y(\bar t))$ is not invertible and hence $y(\bar t)=y_c$. Adding up, $y(t)$ is a critical point of $F$ if and only if $h'(t)=0$.

As $DF(y(t)) = I - H(y(t))T$ for $H(y(t))=DP(Ty(t)) \in \Hab$,  eigenvectors  $\phi(y(t))>0$ and  $\phi^\ast(y(t))>0$ are associated with $r(H(y(t)))$. Use brackets for evaluation of dual vectors,
\[ \langle \phi^\ast(y(t)) , DF(y(t)) y'(t) \rangle = h'(t) \langle \phi^\ast(y(t)), \phi \rangle  \]
and $  h'(t) \langle \phi^\ast(y(t)), \phi \rangle = \langle DF(y(t))^\ast \phi^\ast(y(t)) ,  y'(t) \rangle = \lambda(y(t)) \langle  \phi^\ast(y(t)) ,  y'(t) \rangle  $.

\smallskip
As $\langle \phi^\ast(y(t)), \phi \rangle > 0$,  $\langle  \phi^\ast(y(t)) ,  y'(t) \rangle = 0 $ if and only if $h'(t)=0$.  From the first part of the proof, $t = t_c$ for a critical point $y(t_c)$ and then $\langle  \phi^\ast(y(t_c)) ,  y'(t_c) \rangle = 0 $, which cannot happen because $y'(t_c)>0$. Thus $\langle  \phi^\ast(y(t)) ,  y'(t) \rangle > 0 $
and the quotient $\alpha(t)={\langle  \phi^\ast(y(t)) ,  y'(t) \rangle }/{\langle \phi^\ast(y(t)), \phi \rangle}$ is strictly positive.
\qed
\end{proof}

\begin{prop}\label{fibercrit} Every fiber contains points $y^-$, $y^0$ and $y^+$ for which $r(y^-) < r(y^0)=1 < r(y^+)$, or equivalently, $\lambda(y^-) > \lambda(y^0)=0 > \lambda(y^+)$.
\end{prop}

\begin{proof}
Let $z^0 \in Y$ be critical. For $\psi >0$, set $z^\pm = z^0 + s^\pm \psi$, where $s^- < 0 < s^+$.  By monotonicity (Lemma \ref{smoothlambda}), $r(z^-) < r(z^0)=1 < r(z^+)$.
From Theorem \ref{theo:F} (2), the height function $h(t)$ of each fiber is either strictly increasing or unimodal. From properness (Proposition \ref{properl}), all functions $h(t)$ have the same asymptotic behavior (Proposition \ref{prop:STZ}).

If $h$ along a fiber is strictly decreasing (hence along all fibers), for example, we must have $h'(t) \le 0$. But from Proposition \ref{handr}, $\lambda(y(t)) \le 0 $, or $r(y(t)) \ge 1$, and this violates the existence of a point $z^-$ for which $r(z^-) < 1$. Similarly $h$ cannot be strictly increasing.
We are left with folds, implying the result.
\qed
\end{proof}

\smallskip
We extend Proposition \ref{NT}. A point $y \in Y$ is {\it regular} if it is not critical.

\begin{prop} \label{NT2} Let $y_c \in Y$ be a critical point of $F$,  $r(y_c,y_c)=r(J(Ty_c)T)=1$. Then $F(y_c)$ has a single preimage (clearly $y_c$ itself).
For a regular point $y \in Y$, the point $F(y)$ only has regular preimages.
\end{prop}
\begin{proof} Suppose $F(y) = F(y_c)$, with $y\neq y_c$. Then, as in Proposition \ref{NT}, trichotomy holds: $y > y_c,$
or $y < y_c$ and $r(y,y_c)=1$. Since $T$ is ergodic (r-special), either $T y > T y_c$ or $Ty < T y_c$. Since by hypothesis $r(y_c,y_c)=1$, from (r-Convs), in both cases $r(y,y_c)\ne 1$. The second statement now is obvious.
\qed
\end{proof}

\medskip
From Proposition \ref{properl}, the map $F= I - PT: Y \to Y$ is proper and clearly $PT:Y \to Y$ is a compact map. Moreover, any point has at most two preimages (Proposition \ref{NT}). The degree $\deg(F)$ is a well defined map, and we compute it by considering indices at regular points. More explicitly, for appropriate $y \in Y$, set $\ind(F,y)= \deg(F, B_\epsilon(y), F(y))$, for small $\epsilon >0$.

\begin{prop} \label{index}
\begin{enumerate}
\item [(i)] For $\lambda(y) = 1 - r(y)$,  $\ind(F,y) = \sgn \lambda(y)$.

\item [(ii)] The function $F:Y \to Y$ is not injective. Also, $\deg(F)=0$.
\item [(iii)] The function $F:Y \to Y$ folds downwards.
\end{enumerate}
\end{prop}

\begin{proof}
(i) From a standard argument using the inverse function theorem, at a regular point $y \in Y$, $\ind(F,y) = \ind(F^L,y)$, where
\[ F^L(z) = F(y) + DF(y)(z-y)=  F(y) + (z-y) - J(Ty)T(z-y)  \ . \]
Recall that, for a compact operator $K$, the degree of  $I - K \in \BY$ is $\parity(K)$, the parity of the number of strictly negative eigenvalues of $I-K$, or equivalently, the parity of the number of eigenvalues of $K$ which are larger than 1.
Thus, for the restricted affine map $F^L$, $\deg(F^L, B_\epsilon(y), F^L(y))= \parity(J(Ty)T)$. By Proposition \ref{r-specialMbgen}, $J(Ty)T$ has zero or one eigenvalue larger than 1, depending if $r(y,y)<1$ or $r(y,y)>1$ and then $\ind(F,y) = \parity(J(Ty)T)=1$ or $-1$, always $\sgn \lambda(y)$.

\medskip
(ii) Let $y_c \in Y$ be critical. For $\psi >0$, set $y^\pm = y_c + s^\pm \psi$, where $s^- < 0 < s^+$.  By monotonicity (Lemma \ref{smoothlambda}), the sign of $s^\pm$ equals the sign of $1 - r(J(Ty^\pm)T)$, so that $\ind(F, y^-) =1$, and $\ind(F, y^+) = -1$.
By Propositions \ref{NT} and \ref{NT2}, the image $F(y)$ of any regular point $y \in Y$ has at most two regular preimages, one being $y$. Call the preimages of $F(y^\pm)$ by $y_1^\pm$ and $y_2^\pm$, where, say, $y_1^\pm = y^\pm$. Then
\[ \deg(F) = \ind(F,y_1^-) + \ind(F,y_2^-) = \ind(F,y_1^+) + \ind(F,y_2^+)\] \[
= 1 + \ind(F,y_2^-) = \ind(F,y_2^+) -1 , \]
and each unknown index is either 1 or $-1$. Then necessarily
\[ \deg(F) = 0, \quad \ind(F,y_2^-)= -1, \quad \ind(F,y_2^+) = 1 \ . \]

(iii) Since $F$ is not injective, consider $y^- \ne y^+$ such that $F(y^+) = F(y^-)$. Both points are regular (Proposition \ref{NT2}) and clearly belong to the same fiber, and we write $y^+ = y(t^+)= w^+ + t^+ \phi, y^- = y(t^-)= w^- + t^- \phi$, with $t^+> t^-$. From (ii), we must have that $\ind(F, y^+)$ and $\ind(F,y^-)$ are opposite, so that $\deg(F)=0$, and then $r(y(t^+))$ and $r(y(t^-))$ re at opposite sides of $1$.

From Proposition \ref{prop:convs}, $(y^+ - y^-) - G(y^+,y^-)(y^+ - y^-) =0$
and either $y^+ - y^- >0$ or $( y^+ - y^-) <0$. But $\langle \phi^\ast , y^+ - y^- \rangle = \langle \phi^\ast , (t^+ - t^-) \phi \rangle >0$, so that $y^+ - y^- >0$. By (r-Convs), $r(y(t^+)) > r(y(t^-))$, so that $r(y(t^+)) >  1 > r(y(t^-))$. From Theorem \ref{theo:F}  (3) and $F$ s not injective (item (ii)), $h$ is a (strictly) unimodal function. From Lemma \ref{handr}, $F$ folds downwards.
\qed
\end{proof}

\medskip
\noindent{\bf Prof of Theorem \ref{theo:F} (4):} Combine Propositions \ref{properl} and \ref{index} (3), as stated in the beginning of the section. \qed

\begin{cor} \label{folds} The critical points of $F:Y \to Y$ are topological folds.
\end{cor}

\begin{proof} It suffices to show that the critical points of $F^a: W \oplus \RR \to W \oplus \RR$ are folds. Indeed, when restricted to each vertical line of the domain, $F^a$ is a proper, unimodal function folding downwards. The construction of local charts leading to the normal form of a (topological) fold is left to the reader.
\qed
\end{proof}

\subsection{New folds from regularity} \label{newfolds}

Functions $\Fb= L - P: X \to Y$ between Sobolev spaces sometimes restrict to functions $\Gb: R \to Z$ between Hölder spaces. This is the case for the Dirichlet Laplacian $L = -\Delta_D$ acting on a bounded domain $\Omega \subset \RR^n$ and a Nemitskii map $P(u) = f(u)$ for a smooth function $f: \RR \to \RR$ (\cite{AP}, \cite{BP}). Here, $n > 1$, and  $X = W^{2,n} \cap W^{1,n}_0, Y = L^n, R = C^{2,\alpha}_0, Z = C^{0,\alpha}$, where we drop the reference to the bounded set $\Omega \subset \RR^n$. In both scenarios, $\Fb$ folds downwards.

Hölder spaces are not reflexive, but there Nemitskii operators have better differentiability properties. The non-normal cone of nonnegative continuous functions in $C^{0,\alpha}(\Omega)$ has a nonempty interior.

There is a feature in the example which was not explored in Theorem \ref{theo:F} (4). As in Section \ref{Nemitskii}, from $L: X \to Y$ we define $T = \iota \gamma (L + \gamma I)^{-1}$, and then systematically operate with the composition $y \mapsto PT(y)$ of functions from $Y$ to itself. Instead,  we interpret $PT$ as the composition $\tilde P:X \to Y$ and $\tilde T:Y \to X$: $\tilde P$ may have better smoothness properties than $P$.

The following amplification of Theorem \ref{theo:F} (4) explores this feature.

\begin{theo}\label{theo:FX} Suppose  $Y$ and $K$ as above, $\iota: X \to Y$ an inclusion. For some $\tilde T \in {\cal B}(Y, X)$, let  $T = \iota \tilde T \in \BY$  an r-special operator for which $r(T)=1, T \phi = \phi > 0$.
\noindent Suppose $\tilde P:X \to Y$ satisfies the following properties, for some
${\bf b} > 0$.
\begin{enumerate}
\item[(r-Hs)] $\tilde P: X \to Y$ is a $C^1$ map such that $\tilde J(x)=D \tilde P(x) \in {\cal B}(X,Y)$ admits an extension $J(x) \in \Hab$.
\item [(r-Convs)]
     Let $y, z \in Y$, $y > z$. Then $ J(\tilde Ty) - J(\tilde Tz) \in \BY$ is  strictly positive.
\item[(r-Crit)] $PT = \tilde P \tilde T: Y \to Y$ is $C^1$ and some $y \in Y$ is {\it critical}: $0 \in \sigma(DF(y))$.
\end{enumerate}
Then $F:Y\to Y,  y \mapsto y - P(Ty)$ folds downwards.
\end{theo}

The proof of Theorem \ref{theo:FX} is identical to the proof of Theorem \ref{theo:F} (4).

\medskip

Theorem \ref{transplantation} provides a situation in which one fold gives rise to another. We present it in the geometric context of Theorem \ref{theo:F}: say $F: Y \to Y$ and $G:Z \to Z$ admit a common expression -- can they share geometric properties?

\medskip
\begin{theo} \label{transplantation} Let  $
Y$  and the map $F = I - PT: Y \to Y$ be as in Theorem \ref{theo:FX}, so that $F$ folds downwards with respect to $\phi = \phi(T) \in Y$. Let $Z$ be a Banach space, with $\upsilon: Z \to Y$ be a dense inclusion and  $G : Z \to Z$ be a $C^1$ restriction of $F$ such that $F \circ \upsilon = \upsilon \circ G $.  Moreover, assume that
(i) \ $TY\subset Z$, \ (ii) \ $P: Z \to Z$, (iii) \ For  $z \in Z$, the maps $DG (z): Z \to Z$ are Fredholm operators of index 0.

Then $G $ folds downwards along $\phi$. \end{theo}

In the example above, $X = W^{2,n} \cap W^{1,n}_0 \hookrightarrow C^{0,\alpha} = Z$, for $\alpha \in (0,1)$. For a Nemitskii map $P$ associated with a smooth function $f$, $\tilde P: X \to Z$ is $C^1$. The remaining hypotheses are standard.

\begin{proof} By (i), $\phi \in Z $. From (ii),
for $z \in Z$, the solutions $y$ of $F(y) = z \in Z$ are necessarily in $Z$.
Indeed, write $F(y) = z$ as
$y =  P(Ty) + z $. As $T: Y \to Z$, by (i), the right hand side is necessarily in $Z$ and $y \in Z$.

Moreover,
$V = \langle \phi \rangle \subset Z \subset Y$ and $\phi^\ast \in  Y^\ast \subset Z^\ast$. Decompose $Y = W_Y \oplus V$, for $W_Y = \ker \phi^\ast$ (as $\phi^\ast \in Y^\ast$) and   $Z = W_Z \oplus V$, for $W_Z = \ker \phi^\ast$ (as $\phi^\ast \in Z^\ast$). The four projections are continuous and, for $z \in Z \subset Y$, both decompositions coincide: if $z = w_Z + t_Z \phi = w_Y + t_Y \phi$, then $w_Z= w_Y$ and $t_Z = t_Y$.

Vertical lines $L_y = \{ y + s \phi, s \in \RR\} \subset Y$ either lie in $Z$ or do not intercept it. Indeed, if $y \in Z$, then $y + s \phi \in Z$, by hypothesis (i). The same happens to fibers: if some point of the fiber $\Lambda_y = F^{-1}(L_y)$ lies in $Z$, then the whole fiber does. Indeed, if $F^{-1}(y + h_0 \phi) \in Z$, we have $y \in Z$, from the implication of (ii) in the beginning of the proof,  and then $F^{-1}(y + h \phi) \in Z$ for all $h \in \RR$. Thus, for $y \in Z$, the fiber $\Lambda_y$ is the same for $F$ and $G$.

By hypothesis, $Y$ is foliated by its fibers $\Lambda_y$ and, from the previous fact, the same occurs to $Z$. Continuity is missing: we  prove that fibers in $Z$ are curves.

From Proposition \ref{adaptedcoordinatesFr}, the map $\Psi_F: W_Y \oplus \RR \to W_Y \oplus \RR, (w, t) \mapsto (F^t(w), t)$ is a diffeomorphism. Its restriction $\Psi_G: W_Z \oplus \RR \to W_Y \oplus \RR, (w, t) \mapsto (G^t(w), t) \in Z$ is clearly a $C^1$ map. It is also a bijection, since inverses of points $z \in Z$ under $F$ lie in $Z$. It is a global diffeomorphism: it suffices to show it is locally so. We proceed to apply the inverse function theorem.

Clearly, it is a matter of verifying the invertibility of $DG^t: W_Z \to W_Z$ for each $t$ fixed. Explicitly, $DG^t(w) v = \Pi (v - DP(Tw+ t T\phi)Tv)$. Consider the inclusion $\iota: W_Z \to Z$, and write $DG^t(w) = \Pi \circ DG(w+ t \phi) \circ \iota$, so that, by (iii), $DG^t(w)$ is a Fredholm operator of index 0. Thus, $DG^t(w)$ is not invertible if and only if it has a nontrivial kernel: let $\psi \in \Ker DG^t(w), \psi \ne 0$. As $DG^t(w)$ is a restriction of $DF^t(w)$, we must have $\psi \in \Ker DF^t(w)$, contradicting the fact that $\Psi_F$ is a diffeomorphism. Thus, $\Psi_G$ is a local, hence global diffeomorpshim.

The fibers of $G$ are those $F$, which folds downwards: $G$ does too.
\qed
\end{proof}

\section{The effect of self-adjointness} \label{section:self-adjoint}

Let $H = L^2(M, d\nu)$ be a real Hilbert space with norm $|\cdot|$ for a $\sigma$-finite measure space $(M,\nu)$. Suppose $L: D \subset H \to H$ is a self-adjoint operator.
We make no reference to the complexifications required for the underlying spectral theory.

The concepts related to positivity in both sections overlap, but we keep the more familiar terms in each context. A function is {\it positive}, $u \ge 0$ if $u(x) \ge 0 \ a.e.$ and $u \ne 0$ and {\it strictly positive}, $u > 0$ if $u(x) > 0 \ a.e.$.
The {it positive cone} $K \subset H$ is the set containing the positive functions and zero.
A bounded operator $A: H \to H$ is {\it positivity preserving} if $Au \ge 0 $ for all  $u \ge 0 $. It is {\it positivity improving} if
for any  $u \ge 0$, we have $Au > 0$ and {\it positively stable} if $Au>0$ whenever $u>0$.

For  appropriate self-adjoint operators $L: D \subset H \to H$ and maps $P:H \to H$, we consider counterparts to Theorem \ref{theo:F} for $F:D \to H, F(u) = Lu - P(u)$.

\medskip

\subsection{m-special operators and fine perturbations} \label{section:perturbation}

\medskip
Let $L: D \subset H \to H$ be a self-adjoint operator. An eigenvalue $\lambda \in \sigma(L)$ of  is  {\it basic} if $\lambda \in \RR$ is an isolated point of $\sigma(L)$ and the associated invariant eigenspace is spanned by a strictly positive eigenvector $\phi \in D, \phi >0$.

Set $ \lambda_m = \min \ \sigma(L)$. A self-adjoint operator $L: D \to H$ is {\it m-special} if
\begin{enumerate}
\item[(m-S)]
$\lambda_m$ is a basic eigenvalue; if $0 \in \sigma(L)$ is an eigenvalue, then $0=\lambda_m$.
\item[(m-PI)] $(L-\mu)^{-1}$ is positivity improving for all $\mu < \lambda_m$.
\end{enumerate}

The operators $L$ or $(L - \mu)^{-1}$ are not required to be compact.

Set  $\mu_m  = \inf \sigma(L)\setminus \{ \lambda_m\}$. An operator $A \in {\cal B}(H)$ is a {\it fine perturbation of} $L$ if, for some $b, p \in \RR$, it belongs to
\[ \HbC =\{A\in \BH  \ \hbox{is symmetric}, \  \| A\| \le b < \mu_m,   \ \hbox{and}  \]\[
    A + p I \in \BH \ \hbox{is positive preserving.} \} \]

We prove that, for a fine perturbation $A$ of the m-special operator $L$, the operator $L - A$ is m-special.
We quote Theorem XIII.44 of \cite{RS} and the m-special case of a result by Faris \cite{F} (also exercise 91, Chapter XIII of \cite{RS}).

\begin{lemma} \label{lemma:44} Let $L: D \to H$ be a self-adjoint operator such that
$\lambda_m = \min \sigma(L)$ is an eigenvalue and $e^{-tL}$ is positivity preserving for $t>0$.
Then $\lambda_m$ is a basic eigenvalue if and only if $L$ satisfies (m-PI), or equivalently,
\item [(m-PI2)] For all $t>0$, the operators $e^{-tL}: H \to H$ are positivity improving.
\end{lemma}

\begin{lemma} \label{lemma:Faris} Let $L: D \to H$ be a self-adjoint operator for which (m-PI) holds and let $A: H \to H$ be a positivity preserving, bounded, symmetric operator. Then (m-PI) (hence (mPI2)) holds for $L-A:D \to H$.
\end{lemma}

\begin{prop} \label{prop:mamenom} Let $L: D \to H$ be an m-special operator and let $A \in \HbC$  be a fine perturbation of $L$. Then $S= L - A: D \to H$ is  $m$-special.
\end{prop}

\begin{proof}  By the Kato-Rellich theorem, $S:D \to H$ is  self-adjoint. Since $\min \ \sigma(L)$ is basic,  a standard argument shows that $\sigma_m=\inf \sigma(S)$ is necessarily a simple eigenvalue associated with a normal eigenvector $\psi_m$.
We now prove that if $0$ is an eigenvalue of $S$, and then then $0 = \lambda_m$. Consider the quadratic form $Q(v) = \langle S v, v \rangle$ for  a normal $v \in D$. We must have $Q(\psi_m)=0$. If $\sigma_m < 0$ there is a normal $\psi \in D$ for which $Q(\psi) <0$ and then $Q \le 0$ in the plane $E$ spanned by $\psi_m$ and $\psi$. For a normal $v \in W_D = \phi_m^\perp \subset D$, $\langle L v , v \rangle \ge \mu_m$ and $Q(v) \ge \mu_m - \sup \sigma(A) >0$. But $E$ must intersect the (closed) codimension one subspace $W_D$, a contradiction. This settles (m-S) for $S$.

We now prove (m-PI2).
For $t >0$,  $e^{-tT}$ is positivity improving by (m-PI) for $L$. For some $p$, $ p I + A$ is positive preserving. From Lemma \ref{lemma:44}, $e^{-tS} = e^{-tpI} e^{-t(S - pI)} = e^{-tp} e^{-t(L- (A + pI))} $
 is positivity improving. \qed
\end{proof}

\medskip
\subsection{Counterparts to Theorem \ref{theo:F}} \label{theom}

\medskip
Let $\phi_m >0$ be the normal eigenvector associated with $\lambda_m=\lambda_m(L)$. Set $V = \{ t \phi_m, t \in \RR\}$ and consider the orthogonal decompositions $H = W \oplus V , D = W_D \oplus V$, for $W_D = (W \cap D) $. Let $\Pi$ be the orthogonal projection $\Pi: H \to W$.

\medskip
As in the previous section,
a {\it linearization} of $F=L - P$  is a bounded, symmetric operator $G(u,v) \in \BH$ such that, for $u, v \in D \subset H$, $F(v)-F(u)= (L - G(v,u))(v-u)$. Again, if $P:H \to H$ is a $C^1$ map, linearizations are obtained from Jacobians $J(u)=DP(u): H \to H \in \BH$,
\begin{equation} \label{G2}
F(v)-F(u)=
\big( L - \int_0^1 DP(u + t(v-u)) \ dt \big) (v-u) := (L - G(v,u))(v-u) \ .
\end{equation}

\medskip
Consider the following properties for the Lipschitz map $P: H \to H$.
\begin{enumerate}
\item[(m-H)] There are $b , p \in \RR$ such that, for $u,v \in D$, $\Pi G(u,v) \in \HbC$.
\item[(m-Hs)] There are $b , p \in \RR$ such that, for $u \in D$, $\Pi DP(u) \in \HbC$.
\item[(m-Conv)] If $u<v<w$, then $  \langle w - v , F(v) - F(u) \rangle < \langle v - u , F(w) - F(v) \rangle$ or, equivalently,
$ \langle v-u,P(w) \rangle+\langle w-v, P(u) \rangle + \langle u-w,P(v) \rangle > 0 $.
\item [(m-Convs)] $P:H \to H$ is a $C^1$ map,
     $u,v \in H$, $v>u$. Then $ DP(v) - DP(u) \in \BH$ is  positivity preserving.
\item[(m-Crit)] $P:H \to H$ is $C^1$ and some $u \in H$ is {\it critical}: $0 \in \sigma(DF(y))$.
\end{enumerate}

Again,(m-Hs) and (m-Convs) imply (m-H) and (m-Conv) respectively.

\begin{theo}\label{theo:F2} Suppose  $H$ as above, $L: D \subset H \to H$ an m-special operator, and $P:H\to H$ a Lipschitz (or $C^1$) map admitting linearizations $G(u,v)$.  Define the function $F:D \to H,   u \mapsto Lu - P(u)$.

\smallskip
\noindent(1) If $\ [ - \sup_{u,v} \|G(u,v)\|, \ \sup_{u,v} \|G(u,v)\| \ ]\ \cap \ \sigma(L) = \emptyset$, then $F$ is a homeomorphism.

\smallskip
\noindent (2) Suppose  that (m-H), (m-Conv) hold. Then $F$ is a simple map.

\smallskip
\noindent (3) If (m-H), (m-Conv) hold and $F: H \to H$ is proper, then it is a homeomorphism or folds vertically.

\smallskip
\noindent (4) Suppose that (m-Hs), (m-Convs) and (m-Crit) hold. Then $F$ folds downwards.
\end{theo}

There are no restrictions on $b$ besides $b < \mu_m$: Theorem \ref{theo:F} is more demanding.

\smallskip
The proof follows closely the arguments of the previous section.
For $t \in \RR$, set $P^t : W \to W, \ P^t(w) = \Pi P(w + t \phi_m)$ and $F^t: W_D \to W, F^t(w) = \Pi F(w + t \phi_m)$.

\begin{prop} \label{theo:like2}
Suppose that (m-H) holds.
Then, for each $t \in \RR$, the maps $F^t: W_H \to W$ and $\Psi: W \oplus \RR \to W \oplus \RR, \Psi(w,t) = (F^t(w),t)$  are bilipschitz homeomorphisms. Thus
\[ F^a (z,t)= F \circ \Psi^{-1}(z,t) = (z, h^a(z,t))  \ , \]
for a Lipschitz  {\it height function} $h^a: Y = W \oplus V \to \RR$.
The map $F$ admits fibers $\Lambda_z = \{u(t) = w(t) + t \phi_m, t \in \RR\} = F^{-1}(L_z)$ and $F^a(L_z) = \Lambda_z$.

\end{prop}

\medskip
The argument is  standard (\cite{BP}, \cite{BN}, \cite{CTZ1}, \cite{MST2},  \cite{R}, \cite{R2} \cite{TZ}). The positivity of the eigenfunction associated with $\lambda_m$ is  not required, nor the existence of positive maps $A+ p I$ in the definition of $\HbC$.

\medskip
\begin{proof} Since $L$ is m-special and $W_D = V^\perp$, the restriction $L_W: W_D \to W$ is invertible: $\sigma(L_W)= \sigma(L) \setminus\{\lambda_m\}$. To show the invertibility of $F^t:W \to W$ for each $t \in \RR$, we solve $F^t(w)= \Pi L(w + t \phi_m) - \Pi P(w + t \phi_m) = z$, for $z \in W$, or
\[  y = P^t(L_W^{-1}y) + z \quad \hbox{for} \quad y = L_W w \in W \ . \]
From (m-H), the map $K^t(y) = P^t(L_W^{-1}y)$ is a contraction, with constant $c = b/\mu_m$ independent of $t$. Indeed,
\[ | K^t(w_1) - K^t(w_2)| = |\Pi \big( P(L_W^{-1}(w_1 )+ t\phi_m)  - P(L_W^{-1}(w_2 )+ t\phi_m)\big)|\]
\[ = |\Pi G(L_W^{-1} w_1 + t \phi_m,L_W^{-1} w_2 + t \phi_m)L_W^{-1}( w_1 - w_2)|\le \frac{b}{\mu_m}|w_1-w_2| , \]
as $\| L_W^{-1}\|=1/ \mu_m$ and, by (m-H), $\|\Pi G \| \le b < \mu_m$: $K^t$ is a contraction. Set $c = b/ \mu_m$.

Hence the maps $I - K^t: W \to W$ are (uniform) Lipschitz bijections. The inverses $(I - K^t)^{-1}: W \to W$ are also uniformly  Lipschitz with constant $1/(1-c)$: for $z_i=(I - K^t)^{-1}(y_i),\ i=1,2$,
$$|z_1-z_2|\le |y_1-y_2|+|K^t(z_1)-K^t(z_2)|\le |y_1-y_2| +c|z_1-z_2| \ . $$
Thus  $F^t=(I - K^t) \circ L_W$ are also uniformly bilipschitz homeomorphisms.
The map $\Psi=\left(F^t, \ \hbox{Id} \right): D\to H $ is a homeomorphism because of the continuous dependence of the fixed point with respect to $t$.
\qed
\end{proof}

\medskip
\noindent{\bf Remarks}

\smallskip
\noindent 1. To use Proposition \ref{theo:like2} with the hypotheses of Theorem \ref{AP},  translate $F = {\bf F} - \gamma I $, $P = {\bf P} - \gamma I $ with $\gamma = (a+b)/2$.

\smallskip
\noindent 2. In Theorem \ref{AP}, $P(u) = f(u)$ and the values of $f'$ are close to $\lambda_m$: $P(u)$ is roughly a multiplication by $\lambda_m$ (plus a constant). For a general $P$, once the projected maps $\Pi P^t$ are appropriately bounded,
adapted coordinates apply, the behavior of the map $F$ along the vertical axis may be different.

\smallskip
\noindent 3. Under the hypotheses of Theorem \ref{theo:F2} (1),  the estimates in the proof  apply to $F: D \to H$ itself: $F$ is a bilipschitz homeomorphism. This is the usual Dolph-Hammerstein theorem (\cite{D}, \cite{H}).

\medskip

From Proposition \ref{theo:like2}, vertical lines $L_z = \{z + h \phi_m , h \in \RR\} \subset H$, when inverted by $F$, give rise to fibers
$ \Lambda_z = \{u(t) = w(t) + t \phi_m, w(t) \in W_D, t \in \RR\}$.

\medskip
Without loss, after a translation in the range $H$, we assume $P(0)= F(0)=0$. Indeed, hypotheses (m-H) and (m-Conv) are invariant under this operation.

\begin{cor} \label{cor:Lipschitz} Let $\Lambda_z$ be a fiber. For some $C \in \RR$, $ | w(t) | \le C (| z | + | \Pi P (t \phi_m) |)$.
\end{cor}
\begin{proof} Since $F^t: W_D \to W$ are bilipschitz maps  and $F(u(t)) = z + h(t) \phi_m$,
\[ | w(t) - 0| \le C | F^t(w(t)) - F^t(0) | \]  \[= C | \Pi F(w(t) + t \phi_m) - \Pi F( t \phi_m) |
 =C | z - \Pi P (t \phi_m) | \]
and the result follows. \qed
 \end{proof}

\medskip
Adapted coordinates provide examples. For $-\Delta: H^2(\Omega) \cap H_0^1(\Omega) \to H^0(\Omega)$, with smallest eigenvalue $\lambda_m$ and ground state $\phi_m$, consider, for $t = \langle \phi_m, u \rangle$,
\[ F(u) = L u - P(u) = (- \Delta - \lambda_m) u + (I - \Pi) (\sin t) u \ , \]
The change of variables $ z = (- \Delta - \lambda_m) w$ yields $F^a (z, t) = (z, t\sin t )$, an example of a function $F$ with points with infinitely many preimages.

\begin{prop} \label{TR} If $F(u) = F(v)$ and $P$ satisfies (m-Conv), then either $u > v$ or $u = v$ or $u <v$. If (m-Conv) also holds, $F(u) = g$ has at most two solutions.
\end{prop}

\begin{proof} An integral of fine perturbations as in Equation \ref{G2} for $G(u,v)$, is also fine. By Proposition \ref{prop:mamenom}, $L-G(u,v)$ is m-special. If $F(u) = F(v)$, then $\pm(u - v) \in \Ker(L- G(u,v))$ is a positive  eigenfunction or 0, by (m-S): trichotomy holds.

Suppose $F(u) = F(v)= F(w)= g$. By the previous paragraph, we may assume $u < v < w$, and then the first form of (m-Conv) is clearly violated.  \qed
\end{proof}

\medskip
\noindent{\bf Proof of Theorem \ref{theo:F2}:} Item (1) is a remark after Proposition \ref{theo:like2}. Proposition \ref{theo:like2} showed that no additional restrictions on $b>0$ are necessary. Proposition \ref{TR}, in turn, employed (m-Conv) exactly as  Proposition \ref{NT} employed (r-Conv). The proof now follows the arguments in the previous section.
\qed

\subsection{Asymptotics of heights: properness} \label{asyprop}

Informally, if (m-H) holds, the properness of $F$ depends on its behavior along the vertical axis $V = \langle \phi_m \rangle$. Assume $P:H \to H$ Lipschitz and recall $F(0)=0$. Set
\[ F( t \phi_m) = z_{v}(t) + h_{v}(t) \phi_m = \Pi P(t \phi_m)+ \langle \phi_m, F( t \phi_m) \rangle \phi_m \ , \ \hbox{for} \ \ z_{v} \in W \ . \]

\begin{prop} \label{prop:properness} Suppose the function $F: D \to H$  satisfies (m-H) and
\begin{enumerate}
\item[(m-hor)] $  \lim_{ |t| \to \infty} | z_v(t) |_H \  / \ t = 0$.
\item[(m-ver)]For some $c >0$ and large $|t|$,  $|h_v(t)| \ge c |t|$.
\end{enumerate}
Then $F$ is proper.
\end{prop}

From Corollary \ref{cor:Lipschitz}, $ | w(t) | \le C ( | z | + | z_v(t) |)$. Hypothesis (m-hor) implies that fibers $w(t) + t \phi_m$ are asymptotically vertical: indeed,  $w(t)/t \to 0$ for $|t| \to \infty$. Hypothesis (m-ver) imposes a rate of growth for the height $h_v(t)$.

\medskip
\begin{proof} For $z \in W$, consider the fiber $\Lambda_z = \{u(t) = w(t) + t \phi_m\} $, and its image $F(u(t)) = z + h_z(t) \phi_m$. As $L W_D \subset W$,
\[ h_z(t) =  \langle \phi_m , L (w(t) + t \phi_m) - P(w(t) + t \phi_m) \rangle
 = \lambda_m t - \langle \phi_m ,  P(w(t) + t \phi_m) \rangle \]
  \[= \lambda_m t - \langle \phi_m ,  (P(w(t) + t \phi_m) - P(t \phi_m)) \rangle  - \langle \phi_m ,    P(t \phi_m) \rangle \]

Since $P:H \to H$ is Lipschitz, from Corollary \ref{cor:Lipschitz},
\[ |  P(w( t) + t \phi_m) - P(t \phi_m) |/t \le C_1 | w(t) /t | \le
C_2 | \Pi P(t \phi_m) /t | + o(1) \ , \]
for  constants $C_i$ locally uniformly in $z$. From (m-hor), for large $|t|$,
\[ h_z(t)/t = \lambda_m - \langle \phi_m , P(t \phi_m) \rangle  / t + o(1) = h_v(t)/t + o(1) \  . \]
Thus $h_v$ and $h_z$ have the same rate of growth, given by (m-ver).

We prove the equivalent properness  of $F^a = F \circ \Psi^{-1}$. The estimates for $h_z$ translate to $h_z^a = h_z \circ \Psi^{-1}$, as $\Psi: W_D \oplus \RR \to W \oplus \RR $ in Proposition \ref{theo:like2} leaves the vertical component unaltered. For a convergent sequence in the image of $F^a$, $(z_n, s_n) =(z_n, h^a_z(z_n, t_n) ) \to (z_\infty, s_\infty)$. The estimates for $h^a$, being  locally uniform in $z$, imply the boundedness of $\{ t_n\}$: $F^a$ is proper.  \qed
\end{proof}

\medskip
Under the hypotheses of Theorem \ref{theo:F2}, if the asymptotic signs of $h_v(t)$ are different then $F$ is a homeomorphism. If they are equal, $F$ is a  global fold.

\subsubsection{Nemitskii maps yield properness and folds} \label{Nemcomp}

This natural extension of Theorem \ref{AP} is a counterpart of Theorem \ref{theo:BNV2}.

\begin{theo}\label{theo:Nemitskii2} Suppose  $H$ as above, $L: D \subset H \to H$ an m-special operator such that $\phi_m \in L^1$. Let  $P:H\to H$ be a Nemitskii map $P(u) = f(u)$ for  a strictly convex function  $f:\RR \to \RR$  whose Newton quotient satisfies
\[ \  - \infty < a = \inf_{r \ne s} q(r,s) <  \lambda_m < b  = \sup_{r \ne s} q(r,s)< \mu_m \ . \]
Then the function $F:D \to H,   u \mapsto Lu - P(u)$ folds downwards
\end{theo}

The functions in $H$ do not have to be defined on bounded domains.

We obtain the properness of $F$ from Proposition \ref{prop:properness}.

\begin{lemma} \label{nem}
\noindent (i) $P(u) = f(u) - (b+a)u/2 $ satisfies (m-H) and (m-Conv).

\noindent (ii) (m-hor) holds.

\noindent (iii) (m-ver) also holds: $ \lim_{t \to -\infty} h_v(t)/t = \lambda_m- a >0 , \  \lim_{t \to \infty} h_v(t)/t = \lambda_m- b < 0$.
\end{lemma}

\begin{proof} (i) The verification of properties (m-H) and (m-Conv) is by now familiar.
\noindent (ii) For $v \in W_D$ of norm one and compact support (and then $\Pi v = v$),
\[ \| \Pi P(t \phi_m)/t \|_H = \sup_{|v|=1} \ \left|\langle v , \Pi P(t \phi_m)/t \rangle\right| = \sup_{|v|=1} \ \left|\langle v ,  f(t \phi_m)/t \rangle\right| \ .\]
We use the dominated convergence theorem. We consider $t \to \infty$, the other case is similar. By the convexity of $f$, for $x>0$, $f(tx)/t \to bx $  if $t \to \infty$, and
the integrand $|v f(t \phi_m)/t |$ is pointwise bounded by an $L^1$ function.  Taking limits,
\[ \| \lim_{t \to \infty} \Pi P(t \phi_m)/t \|_H  = \sup_{|v|=1} \ \big| \int_\Omega \lim_{t \to \infty}  v f(t \phi_m)/t \big| = \sup_{|v|=1} \  \big|b \int_\Omega  v \  \phi_m \big| \ =0 ,\]
since  $v \in W_D = V^\perp$. Statement (iii) also follows from dominated convergence.
\qed \end{proof}
\medskip

\noindent{\bf Proof of Theorem \ref{theo:Nemitskii2}:} Combine Theorem \ref{theo:F2} (3) with Lemma \ref{nem} (iii).
\qed

\medskip
A Lipschitz function $f$ induces a Lipschitz Nemitskii map $P: L^2(\Omega) \to L^2(\Omega)$ on  unbounded sets $\Omega \subset \RR^n$ if $f(0) = 0$. As an example of application of Theorem \ref{theo:Nemitskii2}, consider the m-special (see Proposition \ref{examplespecial}) Hamiltonian for the hydrogen atom in $\RR^3$, $L u = - \Delta u - u/r$, with basic eigenvalue $\lambda_m(L)$ associated with a normal, positive ground state $\phi_m \in L^1 \cap L^2$.

\subsection{Some examples}  \label{sec:examples}

Folds related to self-adjoint elliptic operators different from the Dirichlet Laplacian  are known (\cite{B}): we present some more.  The larger class of $m$-special operators $L: D \to H$ includes the  finite dimensional case  $D = H = \RR^n$ when $\nu$ is a finite collection of deltas and the operators described in Proposition \ref{prop:m-special}.

\subsubsection{ m-special operators $L$} \label{examplespecial}

The identification of operators $L$ generating positivity preserving semigroups (hypothesis (m-PI2)) is by itself a field of mathematics.  Arguments in the spirit of Bochner's theorem on distributions of positive type and the Levy-Khintchine formula (Appendix 2, \cite{RS}, vol. IV) lead to a wealth of examples.

\begin{prop} \label{prop:m-special} The following operators are  $m$-special:
I. $-\Delta$ for Dirichlet, Neumann, periodic or mixed boundary conditions on bounded smooth domains; II.
The  hydrogen atom in $\RR^3$, $L u = - \Delta u - u/r$;  III. The quantum oscillator $Lu (x) = - u''(x) + x^2 u(x)$; IV. Fractional powers $L^s, s \in (0,1)$ of positive $m$-special operators.
\end{prop}

\begin{proof}
Hypothesis (m-S) is familiar for all examples, we check (m-PI).
For (I), see Sections 8.1 and 8.2 of \cite{Ar}.
For (II), take $L_0 = -\Delta$ with $D = H^2(\RR^3)$ and define $L = L_0 + V$ for the potential $V = - 1/r$. We prove (m-PI) for $L$ using  Theorem XIII.45, vol. IV of \cite{RS}. Define the bounded truncations $V_n$ which coincide with $V$ for $|x| > 1/n$ and are zero otherwise. Set $q_n = V - V_n$.
Both $L_0$ and $L$ are bounded from below. Comparing quadratic forms, \[ L \le L_0 + V_n \quad \hbox{and} \quad L_0 \le L - V_n \ ,\]
so that $L_0 + V_n$ and $L - V_n$ are uniformly bounded from below. We are left with showing that $L_0 + V_n \to L$ and $L - V_n \to L_0$ in  the strong resolvent sense. By Theorem VIII.25, vol. I of \cite{RS} it suffices to show that, for a given $u \in  H^2$,
\[   q_n \ u \to 0  \quad \hbox{in} \ L^2(\RR^3) \ , \quad \hbox{i.e.} \quad \lim_{\epsilon \to 0} \ \int_{|x| \le \epsilon} \frac{u^2(x)}{|x|^2} \ dx \ = \ 0 \ , \]
which is true, since $H^2(\RR^3)$ consists of bounded, continuous functions. The proof of (III) is similar.
For (IV), use the arguments with Laplace transforms in Section IX.11 of \cite{Yo} (see also \cite{Tay}). \qed
\end{proof}

\medskip
An m-special operator $L: D \to H$ with $\lambda_m >0$ yields  m-special operators $L \otimes L,  L \wedge L, L \otimes I + I \otimes L $ in appropriate domains.

There are natural m-special operators associated to systems. For a
finite measure space $(M,d\nu)$, let $(\tilde M, d\nu)$ be the disjoint union of $k$ copies of  $(M,d\nu)$. Clearly, for $H= L^2(M, d\mu)$, we have $\tilde H = L^2 ( \tilde M,d\nu) \simeq H^k$.  Inequalities between vectors must hold componentwise.
We consider a simple example.

Let $L: D \to H$ be m-special, $\lambda_m(L) = \min \sigma(L)$ associated with  $\phi_m(L)>0$, $\mu_m(L) = \inf \sigma(L) \setminus \lambda_m(L)$.
Then, for $\alpha > 0$ with $-\lambda_m(L) < \alpha < \mu_m(L)$,
\[\tilde L:\tilde{D} = D\oplus D \mapsto \tilde{H} = H \oplus H , \ (u_1,u_2)\mapsto (L u_1-\alpha \  u_2,L u_2- \alpha \  u_1) \]
is m-special. We have $\tilde \lambda_m= \min \sigma(\tilde L)= \lambda_m(L) - \alpha$ and ground state  $\sqrt2 \phi_m=(\phi_m(L),\phi_m(L))$. Also, $\tilde \mu_m = \inf ( \sigma(\tilde L)\setminus \{\tilde \lambda_m\}) = \min \{ \lambda_m(L) + \alpha, \mu_m(L) - \alpha \} > 0$, from which  (m-S) follows;  $\alpha > 0 $ in turn implies (m-PI).

\subsubsection{Compatible maps $P$} \label{section:m-special}

From Proposition \ref{prop:mamenom}, we must find maps $P(u)$ admitting linearizations which are fine perturbations of an m-special operator $L$. We present two examples.

Consider maps $P(u)$ which are gradients of $\Psi: D \to \RR$, so that the Jacobian $DP$ is a Hessian, hence symmetric with appropriate smoothness. Let $\Omega \subset \RR^n$ bounded,  $H = L^2(\Omega, dx)$,  $B, A \in {\cal B}(H)$ and ${\bf f}$ a primitive of $f, {\bf f}'=f$, set
\[ \Psi : H \to \RR \ , \quad u \mapsto \int_\Omega B {\bf f}( A u) \ dx  \]
for which, at least formally (here, $1 \in H$ is the constant function),
\[ P(u) = \nabla \Psi(u) = A^T f(Au) (B^T 1) \ . \]
The usual Nemitskii map is the case $B = A = I$.

\begin{prop} Suppose $A \in \BH$ is positively stable, $g= B^T 1 >0$,  $f$ as in Theorem \ref{theo:Nemitskii2}, $P(u) =  A^T g f(Au)$. Then, if $\| A\|^2 \|gf'\|_\infty  \le b < \mu_m$, (m-H) and (m-Conv) hold and thus $F = L - P$ is a simple map.
\end{prop}

As $f$ is Lipschitz, $f'$ exists a.e. A finer result may be obtained assuming smoothness, as in Proposition \ref{properl}, but we give no details.

\begin{proof} Take for linearizations the maps $G(u,v) = A^T M_{g q(u,v)} A$, where, as usual, $M_{g q(u,v)} \in \BH$ is multiplication by $g q(u,v) \in L^\infty(\Omega)$ and $q$ is the usual Newton quotient associated with $f$. The hypotheses immediately yield (m-H).

We  consider (m-Conv). From the convexity of $f$, for $u < v < w$,
\[ (w-v) (f(v) - f(u)) < (v-u) (f(w) - f(v)) \ . \]
As $A$ is positively stable, $A u < A v < A w$ and
\[   (Aw-Av) g(f(Av) - f(Au)) < (Av-Au) g(f(Aw) - f(Av)) \]
pointwise (for all $x \in \Omega$). Integrate over $\Omega$,
\[  \langle w-v , A^T g(f(Av) - f(Au)) \rangle <  \langle v-u , A^T g(f(Aw) - f(Av)) \rangle ,\]
so that $\langle w-v , P(v) - P(u)  \rangle <  \langle v-u , P(w) - P(v) \rangle$.
\qed
\end{proof}

\medskip
We  consider invariant maps.
For an orthogonal projection $\pi: H \to H$, split
\[ H = H_k \oplus H  ,  \quad  \hbox{for } \ H_k = \Ker \pi, \ H  = \Ran \pi \ , \quad u = u_k + u  \ . \] Suppose $L: D \subset H \to H$ is an  invertible operator  commuting with $\pi$ and $P:H \to H$ is a map keeping $H $ invariant (as a subspace) for which $P(0) = 0$. There are two maps to consider: $F: D \to H, F(u) = Lu - \pi \circ P \circ \pi$ and its restriction $F_\pi: D \cap H  \to H $, which is necessarily of the form $F_\pi (u) = L u - P(u)$.

If the restriction $L_\pi: D \cap H  \to H $ is m-special and $P_\pi: H  \to H $ is $L$-compatible with $L $, Theorem \ref{theo:F2} applies to $F$. We now consider $F$: in order to solve $F(u) = g$, decompose the equation as
\[  L u  - P(u ) = g  , \quad L u_k = g_k \]
and the invertibility of $L$ implies that $F = F \oplus L: H  \oplus H_k \to H  \oplus H_k$, so that the implications of Theorem \ref{theo:F2} and Proposition \ref{prop:properness} are common to $F$ and $F_\pi$.

As a simple example, let $\pi$ be the radial projection on functions defined in a ball, $L = -\Delta$ with Dirichlet conditions and $P$ a Nemitskii map as in Theorem \ref{AP}.

\section{Appendix: Eigenvectors on cones, after Marek}

We will use a standard version of the Krein-Rutman theorem (\cite{De}).

\begin{theo*}[Krein-Rutman] Let $K$ be a reproducing cone of a Banach space $Y$, $S$ a compact operator for which $SK \subset K$ and $r(S) >0$. Then $\lambda (S) = r(S)$ is an eigenvalue associated with an eigenvector $\phi \in K$.
\end{theo*}

Theorem \ref{MSS} below  implies Theorem \ref{theo:KRnointerior} in Section \ref{Cones}.
Theorem \ref{MSS} is a subset of the statements in Theorems 2.1, 2.2 and 2.3 in \cite{Mk}, inspired by ideas in \cite{K}, \cite{Bonsall}, \cite{Sa},  \cite{Sc} and \cite{Sc2}.
Nonsupporting operators and $B$-cones will not be used in this text, and are defined in \cite{Mk}, \cite{Sc2}.

\begin{theo} \label{MSS} [Bonsall-Marek-Sawashima-Schaefer] \label{Marektrue} Let $Y$ be a real Banach space and $K \subset Y$ be a normal B-cone of $Y$. Let $S \in \BY$ be a positive, nonsupporting operator and suppose that the spectral radius $r(S)$ is a pole of the resolvent map $\lambda \mapsto R(\lambda,S) = (\lambda I - S)^{-1}$.
Then the following properties hold.
\begin{enumerate}
\item [(i)]There exists an eigenvector $\phi \in K$ associated with $r(S)$, $\phi >0$.
\item [(ii)]The spectral radius $r(S)$ is a simple pole of $R(\lambda,S)$.
\item [(iii)]Every eigenvector of $S$ in $K$ is a multiple of $\phi$.
\item [(iv)]For an eigenvalue $\lambda \in \sigma(S)$ different from $r(S)$, $|\lambda| < r(S)$.
\item [(v)]There is an eigenvector $\phi^\ast \in K^\ast$ of $S^\ast$ associated with $r(S)$, $\phi^\ast >0$.
\item [(vi)]If $\mu > r(S)$, then $(\mu I - S)^{-1}\in \BY$ and $(\mu I - S)^{-1} K \subset K$.
\end{enumerate}
\end{theo}

We prove Theorem \ref{theo:KRnointerior} from this theorem.
In Section \ref{Cones}, $Y$ is a real Banach space, $K \subset Y$ is a normal, reproducing cone.
A reproducing cone is unflattened (from the Krein-Smulian theorem, \cite{Vu}) and unflattened cones are B-cones (\cite{Sc2}). Also,  $S \in\BY $ is compact and ergodic with $r(S) > 0$. Ergodicity implies that $S$ is positive, nonsupporting. From the Krein-Rutman theorem, $r(S) >0$ and the spectral theory of compact operators implies that $r(S)$ is a pole of the resolvent map. Hence, the hypotheses of Theorem \ref{theo:KRnointerior} imply those of Theorem \ref{MSS}.

We show that $r$ is a basic eigenvalue of $S$.
From (1) and (5), $S$ and $S^\ast$ have eigenvectors $\phi_m >0$ and $\phi_m^\ast > 0$ associated with $r(S)$: this is property (b-1).
By the compactness of $S$ and $S^\ast$, the invariant subspaces $V_{r}$ and $V ^\ast$ associated to $r(S)$ are finite dimensional and it is easy to see that $\dim V  = \dim V ^\ast$. The Jordan theorem for matrices applied to the restriction of $S$ in $V $ implies that the resolvent map $\lambda \to (S - \lambda I)^{-1}$ has a pole at $r(S)$ of multiplicity given by the dimension of $V $. By (2), then, $V $ is one dimensional. This proves properties (b-2) and (b-3): $r(S)$ is a basic eigenvalue of $S$. The remaining claims of Theorem \ref{theo:KRnointerior} are explicit consequences of Theorem \ref{MSS}.

\bigskip
\noindent
Marta Calanchi, Dipartimento di Matematica, Universit\`a degli Studi di Milano, Via Saldini, 50 20147 Milano, Italy. \hfill marta.calanchi@unimi.it

\bigskip
\noindent
Carlos Tomei, Departamento de Matemática, PUC-Rio R. Mq.  S. Vicente 225,
Rio de Janeiro 22451-900,  Brazil. \hfill carlos.tomei@gmail.com

\end{document}